\documentclass{article}
\usepackage{amsfonts}
\usepackage{amsmath}
\usepackage{xcolor}
\usepackage{graphicx}
\newtheorem{theorem}{Theorem}

\newtheorem{definition}[theorem]{Definition}

\newtheorem{lemma}[theorem]{Lemma}

\newenvironment{proof}[1][Proof]{\noindent\textbf{#1.} }{\ \rule{0.5em}{0.5em}}
\input{tcilatex}
\begin{document}

\title{\textbf{Model Selection for independent not identically distributed
observations based on R\'{e}nyi's pseudodistances}}
\author{Angel Felipe$^{1},$ Maria Jaenada$^{1},$ Pedro Miranda$^{1}$ and Leandro Pardo$^{1}$ \\
$^{1}${\small Department of Statistics and O.R., Complutense University of
Madrid, Spain}}
\date{}
\maketitle

\begin{abstract}
Model selection criteria are rules used to select the best statistical model among a set of candidate models, striking a trade-off between goodness of fit and model complexity. Most popular model selection criteria measure the  goodness of fit trough the model log-likelihood function, yielding to non-robust criteria.
This paper presents a new family of robust model selection criteria for independent but not identically distributed observations (i.n.i.d.o.) based on the R\'{e}nyi's pseudodistance  (RP).
The RP-based model selection criterion is indexed with a tuning parameter $\alpha$ controlling the trade-off between efficiency and robustness.
Some theoretical results about the RP criterion are derived and the theory is applied to the multiple linear regression model, obtaining explicit expressions of the model selection criterion.
Moreover, restricted models are considered and explicit expressions under the multiple linear regression model with nested models are
accordingly derived.
Finally, a simulation study empirically illustrates the robustness advantage of the method.

\end{abstract}

\section{Introduction}

Consider a set of real-life observations coming from an unknown distribution to be statistically modeled.
Different candidate models may be assumed to fit the data and so a natural question arises as to how to choose the model that best fits the data.
If the assumed model is too simple, with few number of parameters, it may not capture some important patterns and relationships in the data.
In contrast, if the assumed model is too complex with large number of parameters, the estimated model parameters may over-fit the observed data (including possible sample noise), then resulting in a poor performance when the model is applied to new data.
A model selection criterion is a rule used to select a statistical model among a set of candidates based on the observed data. It defines an objective criterion function quantifying the compromise between goodness of fit and model complexity, typically measured through an expected dissimilarity or divergence. Then, the dissimilarity measure needs to be minimized to select the model with the best trade-off.
In other words, model selection criteria rely on a measure of fairness between a candidate model and the true model (i.e., the probability distribution generating the data).

The Akaike information criterion (AIC)  is one of the most widely known and used in statistical practice model selection criterion.
It was developed by Akaike \cite{aka73, aka74} as the first model selection criterion in the statistical literature. The AIC estimates the expected Kullback-Leibler divergence \cite{kule51} between the true model underlying the data and a fitted candidate model, and selects the model with minimum AIC. Of course, the true model underlying the data is generally unknown and so an empirical estimate obtained from the observed data is used.

Following similar ideas than the AIC, several other model selection criteria have been proposed in the literature.
For example, Schwarz in \cite{sch78} developed the ``Bayesian information criterion" (BIC), which  imposes a stronger penalty for model complexity than AIC.
Also derived from AIC, Hurvich and Tsai \cite{huts89, huts93, huts95} studied the bias problem of the AIC and corrected it with a new criterion called ``Corrected Akaike information criterion" (AIC$_{C}$).
This criterion tries to cope with the fact that the AIC is only asymptotically unbiased and hence, the bias may be important when the sample size is not large enough and the number of parameters is large. Indeed, under small samples sizes the AIC tends to overfitting the observed data.
 Konishi and Kitagawa \cite{koki96} extended the framework in which AIC has been developed to a general framework, including other estimation methods than maximum likelihood  to fit the assumed candidate model. The resulting model selection criterion was called the ``generalized information criterion" (GIC).
 The penalty term of GIC reduces to that of ``Takeuchi information criterion" (TIC) developed by Takeuchi in \cite{tak76} when the fitting method is maximum likelihood.
 Finally, Bozdogon \cite{boz87} proposed another variant of AIC, called CAIC, that corrected its lack of consistency.  Interesting surveys about model selection criteria can be found in \cite{rawu01, cane11}.

Most of the previous procedures measure the fairness in terms of the Kullback-Leibler divergence. However, some other divergence measures have been explored, extending the methods with better robustness properties.
For example, \cite{maleka09} considered the density power divergence (DPD) \cite{bahahjjo98}  to define a robust model selection criterion. Similarly, Toma et al. \cite{tokatr20} introduced  another robust criterion for model selection based on the R\'{e}nyi pseudodistance (RP) \cite{johjhaba01}.

All the previous criteria assume that the observations are independent and identically distributed. A new problem appears if the observations are independent but not identically distributed (i.n.i.d.o.). In this context, Kurata and Hamada \cite{kuha18} considered a criterion based on DPD, extending the theory of \cite{maleka09}. The main purpose of this paper is to introduce a new robust model selection criterion in the context of i.n.i.d.o. based on RP, thus extending the methods of \cite{tokatr20}.


The rest of the paper goes as follows. In Section 2 we introduce RP for i.n.i.d.o. and we present some theoretical results necessary for next sections. The criterion based on RP is considered in Section 3 and an application to multiple linear regression model (MLRM) is presented.  Section 4 studies the restricted case, where some additional conditions on the parameter space are imposed. The corresponding explicit expressions for the MLRM comparing a model with many parameters to other with a reduced number of parameters are derived. In Section 5 a simulation study illustrates the robustness of the proposed criterion and compare it with other model selection criteria. Section 6 deals with a real data example. Some final conclusions are presented in Section 7.

\section{R\'{e}nyi's pseudodistance for independent but not identically distributed observations}

 Let $Y_{1},...,Y_{n}$ be i.n.i.d.o. observations, where each $Y_i$ has true probability distribution function $G_{i}, i=1,...,n,$ and probability density function $g_{i}, i=1,...,n,$ respectively.
 For inferential purposes, it is assumed that the true density function $g_{i}$ could belong to a parametric family of densities,
  $f_{i}(y,\boldsymbol{\theta }), i=1,...,n,$ with $\boldsymbol{\theta } \in \Theta \subset \mathbb{R}^{p}$  a common model parameter for all the density functions. 
In the following, we shall denote by $F_{i}(y, \boldsymbol{\theta })$ the distribution function associated to the density function $f_{i}(y,\boldsymbol{\theta }), i=1, ..., n.$

The value of $\boldsymbol{\theta}$ that best fits the original distributions $g_1, ..., g_n$, would  naturally minimize  some kind of distance between the true and assumed densities, $(g_1(y), ..., g_n(y))$ and $(f_1(y, \boldsymbol{\theta }), ..., f_n(y, \boldsymbol{\theta })).$
Here, we will use the family of RP divergence measures defined in \cite{johjhaba01} as measure of closeness between both sets of densities.

\begin{definition}
Consider $f(\cdot ,\boldsymbol{\theta }), g(\cdot)$ two probability density functions. The {\bf R\'enyi's pseudodistance} (RP) between $f$ and $g$ of tuning parameter $\alpha >0$ is defined by

\begin{equation}\label{2.1}
	\begin{aligned}
		R_{\alpha }\left( f(\cdot ,\boldsymbol{\theta }),g(\cdot )\right) =
		&\frac{1}{	\alpha +1}\log \left( \int f(y,\boldsymbol{\theta })^{\alpha+1}dy\right)
		-\frac{1}{\alpha }\log \left( \int f(y,\boldsymbol{\theta })^{\alpha }g(y)dy\right)\\
		& +\frac{1}{\alpha \left( \alpha +1\right) }\log \left( \int g(y)^{\alpha
			+1}dy\right).
	\end{aligned}
\end{equation}
\end{definition}

The tuning parameter $\alpha $ controls the trade-off between efficiency and robustness. Hence, for small values of $\alpha $ (in the limit $\alpha =0$), the corresponding results will be more efficient while less robust. On the other hand, for large values of $\alpha $, the results will lead to robustness but with a loss of efficiency. 
 
The RP divergence defined in Eq. (\ref{2.1}) is always positive and it  only reaches the zero when both densities coincide. Then,
 the best model parameter value approximating the underlying distribution would naturally minimize Eq. (\ref{2.1}) in $\boldsymbol{\theta }\in \boldsymbol{\Theta }.$ Indeed,  if the true distribution $g$ belongs to the assumed parametric model with true parameter $\boldsymbol{\theta}_0$,  the global minimizer of the RP is necessarily $\boldsymbol{\theta} = \boldsymbol{\theta}_0.$

At $\alpha =0,$ the corresponding {\bf R\'enyi's pseudodistance} between $f$ and $g$ can be defined by taking continuous limits as follows
\begin{eqnarray}
	R_{0}\left( f(\cdot ,\boldsymbol{\theta }),g(\cdot )\right) & = & \lim_{\alpha \downarrow 0}R_{\alpha }\left( f(y,\boldsymbol{\theta }),g(y)\right) = \int g(y) \log {g(y)\over f(y, \boldsymbol{\theta })} dy \nonumber \\
	& = & \int g(y) \log g(y) dy - \int g(y) log f(y,\boldsymbol{\theta }) dy.
	\label{2.3}
\end{eqnarray}

Hence, $R_0\left( f(\cdot ,\boldsymbol{\theta }),g(\cdot )\right) $ coincides with the Kullback-Leibler divergence measure between $g$ and $f$.
The RP have been applied in many different statistical models with very promising results in terms of robustness with a small loss of efficiency.
For example,  \cite{fueg08} considered the RP divergence under the name of $\gamma $-cross entropy. Additionally, Toma and Leoni-Auban \cite{tole10} defined new robust and efficient measures based on RP.
In \cite{cajapa22}, Wald-type tests based on RP were developed in the context of MLRM, and were extended later in \cite{japa22} for the generalized multiple regression model. Moreover, in \cite{japa22} a robust approach for comparing two dependent normal populations via a Wald-type test based on RP was carried out. In \cite{jamipa22} the restricted MRPE was considered and their asymptotic properties studied; moreover, an application to Rao-type tests based on the restricted RP was there developed.

 Note that the last term in Eq. (\ref{2.1})
does not depend on $\boldsymbol{\theta }.$ Hence, the minimizer of the RP measure can be obtained,
 for $\alpha >0,$  
 by minimizing the surrogate function
\begin{equation}
\frac{1}{\alpha +1}\log \left( \int f_{i}(y,\boldsymbol{\theta })^{\alpha
+1}dy\right) -\frac{1}{\alpha }\log \left( \int f(y,\boldsymbol{\theta })^{\alpha } g(y) dy \right) . \label{2.2}
\end{equation}
The above expression can be rewritten using logarithm properties as
\begin{equation*}
- \frac{1}{\alpha }\log \frac{\int f(y,\boldsymbol{\theta })^{\alpha } g(y)dy}{\left( \int f(y,\boldsymbol{\theta })^{\alpha +1}dy\right) ^{\frac{\alpha }{\alpha +1}}},
\end{equation*}
and thus minimizing $R_{\alpha }(f(\cdot , \boldsymbol{\theta }), g(\cdot )) $ in $\boldsymbol{\theta }$, for $\alpha >0,$ is
equivalent to minimize
\begin{equation}\label{Eq3}
V_{\alpha}^\ast\left( \boldsymbol{\theta }\right) = - \frac{\int f(y,\boldsymbol{\theta })^{\alpha } g(y) dy}{\left( \int f(y,\boldsymbol{\theta })^{\alpha +1}dy\right)
^{\frac{\alpha }{\alpha +1}}}.
\end{equation}

Similarly, for $\alpha =0,$ we have that the first term in Eq. (\ref{2.3}) does not depend on $\boldsymbol{\theta}$ and hence, minimizing $R_{0}\left( f(\cdot ,\boldsymbol{\theta }),g(\cdot )\right) $ is equivalent to minimizing
\begin{equation}\label{Nueva}
V^*_{0}\left( \boldsymbol{\theta }\right) = - \int g(y) log f(y,\boldsymbol{\theta }) dy.
\end{equation}

However, now Expression (\ref{Eq3}) does not tend to Expression (\ref{Nueva}) when $\alpha \rightarrow 0.$
In order to recover such convergence, and then extend the classical results based on Kullback-Leibler divergence, we slightly modify Expression (\ref{Eq3}) as
\begin{equation}\label{Aux}
V_{\alpha}\left( \boldsymbol{\theta }\right) = - \frac{\int f(y,\boldsymbol{\theta })^{\alpha } g(y) dy}{\alpha \left( \int f(y,\boldsymbol{\theta })^{\alpha +1} dy\right)
^{\frac{\alpha }{\alpha +1}}} + {1\over \alpha},
\end{equation}
where the value of $\boldsymbol{\theta}$ minimizing (\ref{Eq3}) is the same as for minimizing (\ref{Aux}). 
Next lemma proves the required convergence of the objective functions.

\begin{lemma} For any two density function $f(\cdot ,\boldsymbol{\theta })$ and $g(\cdot ),$ the following convergence holds
$$ \lim_{\alpha \rightarrow 0} V_{i,\alpha }(\boldsymbol{\theta })= V_{i,0}(\boldsymbol{\theta }). $$
\end{lemma}

\begin{proof}
First, note that

\begin{equation}\label{limite}
\lim_{\alpha \rightarrow 0}   \left( - \frac{\int f(y,\boldsymbol{\theta })^{\alpha } g(y) dy}{\alpha \left( \int f(y,\boldsymbol{\theta })^{\alpha +1} dy\right)
^{\frac{\alpha }{\alpha +1}}}  + {1\over \alpha} \right)
\end{equation}
 leads to an indeterminate $(0/0)$. Let us denote

$$ z(\alpha )= \left( \int f(y,\boldsymbol{\theta })^{\alpha +1} dy\right)
^{\frac{\alpha }{\alpha +1}}.$$

Taking derivatives on its logarithm
$$ \log z(\alpha )= {\alpha \over \alpha +1} \log \left( \int f(y,\boldsymbol{\theta })^{\alpha +1} dy\right),$$
we obtain, after some algebra, that
 ${\partial \log z(\alpha )\over \partial \alpha } = {1\over z(\alpha )} {\partial z(\alpha )\over \partial \alpha }.$
On the other hand, the derivative of the function $ \log z(\alpha )$ is given by
$$ {\partial \log z(\alpha )\over \partial \alpha } = {1\over (\alpha +1)^2} \log \left( \int f(y,\boldsymbol{\theta })^{\alpha +1} dy\right) + {\alpha \over \alpha +1} {\left( \int f(y,\boldsymbol{\theta })^{\alpha +1} \log f(y,\boldsymbol{\theta }) dy\right) \over \left( \int f(y,\boldsymbol{\theta })^{\alpha +1} dy\right)},$$
and solving the above equation we have that
\begin{align*}
	 {\partial z(\alpha )\over \partial \alpha } = &\left[ {1\over (\alpha +1)^2} \log \left( \int f(y,\boldsymbol{\theta })^{\alpha +1} dy\right) + {\alpha \over \alpha +1} {\left( \int f(y,\boldsymbol{\theta })^{\alpha +1} \log f(y,\boldsymbol{\theta }) dy\right) \over \left( \int f(y,\boldsymbol{\theta })^{\alpha +1} dy\right)} \right]\\
	 & \times  \left( \int f(y,\boldsymbol{\theta })^{\alpha +1} dy\right)
	^{\frac{\alpha }{\alpha +1}}.
\end{align*}

Hence, applying L'H\^{o}pital rule in (\ref{limite}), we obtain that
$$ \lim_{\alpha \rightarrow 0} - \frac{\int f(y,\boldsymbol{\theta })^{\alpha } g(y) dy}{\alpha \left( \int f(y,\boldsymbol{\theta })^{\alpha +1} dy\right)
^{\frac{\alpha }{\alpha +1}}}  + {1\over \alpha} = \lim_{\alpha \rightarrow 0} { - \int f(y,\boldsymbol{\theta })^{\alpha } g(y) \log f(y,\boldsymbol{\theta }) dy + {\partial z(\alpha )\over \partial \alpha } \over z-\alpha {\partial z(\alpha )\over \partial \alpha }}.$$

Finally,

\begin{itemize}
\item $ \lim_{\alpha \rightarrow 0} \int f(y,\boldsymbol{\theta })^{\alpha } g(y) \log f(y,\boldsymbol{\theta }) dy = \int g(y) \log f(y,\boldsymbol{\theta }) dy.$
\item $ \lim_{\alpha \rightarrow 0} {\partial z(\alpha )\over \partial \alpha } = {1\over 1} \log 1 + {0\over 1} {\int f(y,\boldsymbol{\theta }) \log f(y,\boldsymbol{\theta }) dy \over 1} =0.$
\item $ \lim_{\alpha \rightarrow 0} z  = 1^0 = 1.$
\end{itemize}
Hence, the result holds.
\end{proof}

Now, let us denote $V_{i, \alpha }(\boldsymbol{\theta })$
the corresponding objective functions for each pair of distributions $(f_{i}(y,\boldsymbol{\theta }), g_i(y)), i=1,...,n,$ as given in (\ref{Aux}).
As all densities $f_{i}(y,\boldsymbol{\theta })$ share a common parameter, the model parameter that best approximates the different underlying densities should minimize the weighted objective function, giving equal weighting to all functions $V_{i, \alpha }(\boldsymbol{\theta }).$ Hence, we consider
\begin{equation}\label{A}
H_{\alpha }(\boldsymbol{\theta })= {1\over n} \sum_{i=1}^n V_{i, \alpha }(\boldsymbol{\theta })= {1\over n} \sum_{i=1}^n\left[-\frac{\int f_{i}(y,\boldsymbol{\theta })^{\alpha } g_i(y) dy}{\alpha \left( \int f_{i}(y,\boldsymbol{\theta })^{\alpha +1}dy\right)^{\frac{\alpha }{\alpha +1}}}  + {1\over \alpha} \right].
\end{equation}
%

\begin{definition}
Consider $(g_1(y), ..., g_n(y))$ and $ (f_1(y, \boldsymbol{\theta }), ..., f_n(y, \boldsymbol{\theta })),$ $n$ pairs of true and assumed densities for i.n.i.d.o. random variables $Y_i, i = 1,...,n.$
 For any $\alpha \geq 0,$ the value $\boldsymbol{\theta }_{\boldsymbol{g},\alpha }$ satisfying
\begin{equation*}
\boldsymbol{\theta }_{\boldsymbol{g}, \alpha }=\arg \min_{\boldsymbol{\theta }} {1\over n} \sum_{i=1}^n \left[-\frac{\int f_i(y,\boldsymbol{\theta })^{\alpha } g_i(y) dy}{\alpha \left( \int f_i(y,\boldsymbol{\theta })^{\alpha +1}dy\right)
^{\frac{\alpha }{\alpha +1}}}  + {1\over \alpha} \right]= \arg \min_{\boldsymbol{\theta }} {1\over n} \sum_{i=1}^n V_{i, \alpha }(\boldsymbol{\theta}).
\end{equation*}
is called the {\bf best-fitting parameter according to RP}.

\end{definition}

In the following we shall assume that there exists an open subset $\boldsymbol{\Theta _{0}}\subset \boldsymbol{\Theta }$ that contains the best-fitting parameter $\boldsymbol{\theta }_{\boldsymbol{g},\alpha }.$

For any fixed $i=1,...,n,$ the true distribution $g_i$ of the random variable $Y_i$ is usually unknown in practice and thus $\boldsymbol{\theta }_{\boldsymbol{g}, \alpha }$ must be empirically estimated. As we only have one observation of each variable $Y_{i},$ the best way to estimate $g_{i}$ based on the observation $y_i$ is assuming that the distribution is degenerate in $y_{i}.$ We will denote this degenerate distribution by $\widehat{g}_i.$ Therefore, the empirical estimate of the RP divergence with $\alpha >0$, given in Eq. (\ref{2.1}) is
\begin{equation}
R_{\alpha }\left( f_{i}(Y_{i},\boldsymbol{\theta }),\widehat{g}_{i}\right) = \frac{1}{\alpha +1}\log \left( \int f_{i}(y,\boldsymbol{\theta })^{\alpha +1}dy\right) -\frac{1}{\alpha }\log f_{i}(Y_{i},\boldsymbol{\theta })^{\alpha }+k,  \label{2.7}
\end{equation}
and similarly the empirical estimate of the RP for $\alpha=0,$ stated in (\ref{2.3}), yields to
\begin{equation}\label{2.8}
R_{0}\left( f_{i}(Y_{i},\boldsymbol{\theta }),\widehat{g}_{i}\right) = - \log f_{i}(Y_{i},\boldsymbol{\theta })+k,
\end{equation}
where $k$ in (\ref{2.7}) and (\ref{2.8}) denotes a constant that does not depend on $\boldsymbol{\theta}.$
As discussed earlier, the best estimator of the model parameter $\boldsymbol{\theta},$ based on the  RP divergence should minimize its empirical estimate.
%
But again, minimizing the estimated RP, $R_{\alpha }\left( f_{i}(Y_{i},\boldsymbol{\theta }),\widehat{g}_{i}\right), $  for $\alpha >0,$ is equivalent to minimizing
\begin{equation}\label{V_i}
\widehat{V}_{i,\alpha}\left(Y_i, \boldsymbol{\theta }\right) = -\frac{f_{i}(Y_{i},\boldsymbol{\theta })^{\alpha }}{\alpha \left( \int f_{i}(y,\boldsymbol{\theta })^{\alpha +1}dy\right)
^{\frac{\alpha }{\alpha +1}}} +{1\over \alpha }.
\end{equation}
and for $\alpha =0,$ we can proceed the same way and conclude that minimizing $R_{0}\left( f_{i}(Y_{i},\boldsymbol{\theta }),\widehat{g}_{i}\right) $ in $\boldsymbol{\theta }$, is equivalent to minimizing
\begin{equation}\label{V_i0}
\widehat{V}_{0,\alpha}\left(Y_i, \boldsymbol{\theta }\right) = - \log f(Y_i,\boldsymbol{\theta }).
\end{equation}


Now, all the available information about the true value of the parameter comes from  the set observed data, and so to obtain the best estimation fitting jointly all the observations we should consider the weighted objective function given for for $\alpha >0$ as
\begin{equation}
	\begin{aligned}
		H_{n, \alpha }(\boldsymbol{\theta })=&
		\frac{1}{n} \tsum\limits_{i=1}^{n} \left[-\frac{f_{i}(Y_{i},\boldsymbol{\theta })^{\alpha }}{\alpha L_{\alpha }^{i}\left( \boldsymbol{\theta }\right) } + {1\over \alpha }\right]\\
		 =& \frac{1}{n} \tsum\limits_{i=1}^{n}\widehat{V}_{i, \alpha }(Y_{i},\boldsymbol{\theta }).
	\end{aligned} \label{2.4}
\end{equation}
with
\begin{equation*}
	L_{\alpha }^{i}\left( \boldsymbol{\theta }\right) =\left( \int f_{i}(y, \boldsymbol{\theta })^{\alpha +1}dy\right) ^{\frac{\alpha }{\alpha +1}},
	\label{EL}
\end{equation*}
and correspondingly,
\begin{equation} \label{2.5}
	H_{n, 0}(\boldsymbol{\theta }) = \lim_{\alpha \rightarrow 0} H_{n, \alpha }(\boldsymbol{\theta }) = {1\over n} \sum_{i=1}^n \widehat{V}_{i, 0}(Y_{i},\boldsymbol{\theta }).
\end{equation}

Remark at this point that the expected values of the estimates are indeed the theoretical objective functions
$$ V_{i,\alpha }(\boldsymbol{\theta }) = E_{Y_i} \left[ \widehat{V}_{i, \alpha }(Y_i, \boldsymbol{\theta })\right] ,\quad H_{\alpha }(\boldsymbol{\theta }) = E_{Y_1, ..., Y_n} \left[ H_{n, \alpha }(\boldsymbol{\theta })\right] .$$

\begin{definition}
Given $Y_{1},...,Y_{n}$ be i.n.i.d.o. and $\alpha >0,$ the {\bf minimum RP estimator (MRPE)}, $\widehat{\boldsymbol{\theta }}_{\alpha },$ is given by
\begin{equation}
\widehat{\boldsymbol{\theta }}_{\alpha }=\arg \min_{\boldsymbol{\theta } \in \boldsymbol{\Theta }}H_{n, \alpha }(\boldsymbol{\theta }),  \label{2.4.1}
\end{equation}
with $H_{n, \alpha }(\boldsymbol{\theta })$ defined in (\ref{2.4}) for $\alpha >0$ and in (\ref{2.5}) for $\alpha = 0.$
\end{definition}

Note that at $\alpha =0,$ we recover the maximum likelihood estimator (MLE) of the model and so the MRPE family includes the classical estimator as a particular case.

As the MRPE, $\widehat{\boldsymbol{\theta }}_{\alpha },$ is a minimum of a differentiable function, it must annul the first derivatives of the function $H_{n,\alpha}(\boldsymbol{\theta })$
\begin{equation*}
\frac{1}{n}\tsum\limits_{i=1}^{n}\frac{\partial \widehat{V}_{i, \alpha }(Y_{i};\boldsymbol{\theta })}{\partial \theta _{j}}=0, \hspace{0.3cm} j=1, ...,p.
\end{equation*}

That is, the estimation equations of the MRPE are
\begin{equation*}
\frac{1}{n}\tsum\limits_{i=1}^{n}\frac{1}{\alpha L_{\alpha }^{i}\left( \boldsymbol{\theta }\right) ^{2}}\left( \alpha f_{i}(Y_i, \boldsymbol{\theta })^{\alpha }u_{j}(Y_i, \boldsymbol{\theta })L_{\alpha }^{i}\left( \boldsymbol{\theta }\right) -\frac{\partial L_{\alpha }^{i}\left( \boldsymbol{\theta }\right) }{\partial \theta _{j}}f_{i}(Y_i, \boldsymbol{\theta })^{\alpha }\right) = 0, \hspace{0.3cm} j=1, ...,p,
\end{equation*}
with
\begin{equation*}
u_{j}(y,\boldsymbol{\theta })=\frac{\partial \log (f_{i}(y,\boldsymbol{\theta }))}{\partial {\theta _{j}}},
\end{equation*}
and
\begin{align*}
\frac{\partial L_{\alpha }^{i}\left( \boldsymbol{\theta }\right) }{\partial \theta _{j}} & = \frac{\alpha }{\alpha +1}\left( \int f_{i}(y,\boldsymbol{\theta })^{\alpha +1}dy\right) ^{\frac{\alpha }{\alpha +1}-1}\left( \alpha
+1\right) \int f_{i}(y,\boldsymbol{\theta })^{\alpha +1}u_{j}(y,\boldsymbol{\theta })dy \\
& =\alpha \left( \int f_{i}(y,\boldsymbol{\theta })^{\alpha +1}dy\right) ^{\frac{\alpha }{\alpha +1}-1}\int f_{i}(y,\boldsymbol{\theta })^{\alpha +1}u_{j}(y,\boldsymbol{\theta })dy, i=1, ..., n.
\end{align*}

It is interesting to observe that if $Y_{1},...,Y_{n}$ are independent and identically distributed (i.i.d.) random variables, the MRPE $\widehat{\boldsymbol{\theta }}_{\alpha }$ coincides with the estimator $\widehat{\boldsymbol{\theta }}_{\alpha }^{\ast }$ proposed in \cite{brtova12}.

We next study the asymptotic distribution of the MRPE, $\widehat{\boldsymbol{\theta }}_{\alpha }.$
For notation simplicity, let us define the matrices $\boldsymbol{\Psi }_{n, \alpha }\left( \boldsymbol{\theta }_{\boldsymbol{g}, \alpha }\right) $ and $\boldsymbol{\Omega }_{n, \alpha }\left( \boldsymbol{\theta }_{\boldsymbol{g}, \alpha }\right) $ as follows:
\begin{equation}
\boldsymbol{\Psi }_{n, \alpha }\left( \boldsymbol{\theta }_{\boldsymbol{g}, \alpha }\right) =\frac{1}{n}\tsum\limits_{i=1}^{n}\boldsymbol{J}_{\alpha }^{\left( i\right) }\left( \boldsymbol{\theta }_{\boldsymbol{g}, \alpha }\right) ,  \label{2.4.2}
\end{equation}
with
\begin{equation*}
\boldsymbol{J}_{\alpha }^{\left( i\right) }\left( \boldsymbol{\theta }_{\boldsymbol{g}, \alpha }\right) =\left( E_{Y_i}\left[ \frac{\partial ^{2}\widehat{V}_{i, \alpha }(Y_i;\boldsymbol{\theta })}{\partial \theta _{j}\partial \theta _{k}}\right] _{\boldsymbol{\theta } = \boldsymbol{\theta }_{\boldsymbol{g}, \alpha }}\right) _{j, k=1,...,p}, i=1, ..., n,
\end{equation*}
and
\begin{equation}
\boldsymbol{\Omega }_{n, \alpha }\left( \boldsymbol{\theta }_{\boldsymbol{g}, \alpha }\right) = \frac{1}{n}\tsum\limits_{i=1}^{n}Var_{Y_i} \left[ \left( \frac{\partial \widehat{V}_{i, \alpha }(Y_i;\boldsymbol{\theta })}{\partial \theta _{j}}\right)_{j=1,..,p}\right] _{\boldsymbol{\theta }=\boldsymbol{\theta }_{\boldsymbol{g}, \alpha }},\, i=1, ..., n. \label{2.4.3}
\end{equation}

Additionally, let $\lambda_1, ..., \lambda_n$ be the eigenvalues of $\boldsymbol{\Omega }_{n, \alpha }\left( \boldsymbol{\theta }_{\boldsymbol{g}, \alpha }\right) .$ From now on, we will assume that
$
\inf_{n}\lambda _{n}>0,
$
so that $\boldsymbol{\Omega }_{n, \alpha }\left( \boldsymbol{\theta }_{\boldsymbol{g}, \alpha }\right) $ can be inverted.

We consider the following regularity conditions:

\begin{description}

\item[C1.\label{itm:C1}] The support, $\mathcal{X},$ of the density functions $f_{i}(y,\boldsymbol{\theta })$ is the same for all $i$ and it does not depend on $\boldsymbol{\theta }.$ Besides, the true probability density
functions $g_{1},...,g_{n}$ have the same support $\mathcal{X}$.

\item[C2.\label{itm:C2}] For almost all $y\in \mathcal{X}$ the density $f_{i}(y,\boldsymbol{\theta })$ admits all third derivatives with respect to $\boldsymbol{\theta \in \Theta }$ and $i=1,...,n.$

\item[C3.\label{itm:C3}] For $i=1,2,..., n$ the integrals

\begin{equation*}
\tint f_{i}(y,\boldsymbol{\theta })^{1+\alpha }dy\text{ }
\end{equation*}%
can be differentiated thrice with respect to $\boldsymbol{\theta }$ and we can interchange integration and differentiation. As a consequence of this condition, it follows that

$$ {\partial V_{i, \alpha }(\boldsymbol{\theta })\over \partial \boldsymbol{\theta }} = E_{Y_i}\left[ {\partial \widehat{V}_{i, \alpha }(Y_i, \boldsymbol{\theta })\over \partial \boldsymbol{\theta }}\right] , \quad
{\partial^2 V_{i, \alpha }(\boldsymbol{\theta })\over \partial \boldsymbol{\theta } \partial \boldsymbol{\theta }^T} = E_{Y_i}\left[ {\partial^2 \widehat{V}_{i, \alpha }(Y_i,\boldsymbol{\theta })\over \partial \boldsymbol{\theta }\partial \boldsymbol{\theta }^T}\right] = \boldsymbol{J}_{ \alpha }^{(i)}(\boldsymbol{\theta }).$$

\item[C4.\label{itm:C4}] For $i=1,2,..., n$ the matrices $\boldsymbol{J}_{\alpha }^{\left( i\right) }\left( \boldsymbol{\theta }_{\boldsymbol{g}, \alpha }\right) $ are positive definite.

\item[C5.\label{itm:C5}] There exist functions $M_{jkl}^{\left( i\right) }$ and constants $m_{jkl}$ such that

\begin{equation*}
\left\vert \frac{\partial ^{3}\widehat{V}_{i, \alpha }(y;\boldsymbol{\theta })}{\partial \theta_{j}\partial \theta _{k}\partial \theta _{l}}\right\vert \leq M_{jkl}^{\left( i\right) }\left( y\right) ,\text{ \qquad }\forall \boldsymbol{\theta }\in \boldsymbol{\Theta },\text{ }\forall j,k,l
\end{equation*}
and
\begin{equation*}
E_{Y}\left[ M_{jkl}^{\left( i\right) }\left( Y\right) \right] =m_{jkl}<\infty ,\text{ \qquad }\forall \boldsymbol{\theta }\in \boldsymbol{\Theta },\text{ }\forall j,k,l.
\end{equation*}

\item[C6.\label{itm:C6}] For all $j,k,l$ and $\boldsymbol{\theta }\in \boldsymbol{\Theta },$ the sequences $\left\{ \frac{\partial \widehat{V}_{i, \alpha }(Y_i, \boldsymbol{\theta})}{\partial\theta_{j}}\right\}_{j=1,...,p}$ ,$\left\{ \frac{\partial^{2}\widehat{V}_{i, \alpha }(Y_i, \boldsymbol{\theta})}{\partial\theta _{j}\partial\theta_{k}}\right\} _{j,k=1,..,p}$ and $\left\{ \frac {\partial^{3}\widehat{V}_{i, \alpha }(Y_i, \boldsymbol{\theta})}{\partial\theta_{j}\partial \theta_{k}\partial l}\right\} _{j,k,l=1,..,p}$ are uniformly integrable in the Ces\`{a}ro sense, i.e.

\begin{align*}
\lim_{n\rightarrow \infty }\left( \sup_{n>1}\frac{1}{n}\tsum \limits_{i=1}^{n}E_{Y_i}\left[ \left\vert \frac{\partial \widehat{V}_{i, \alpha }(Y_i, \boldsymbol{\theta })}{\partial \theta _{j}}\right\vert {\LARGE I}_{\left\{ \frac{\partial V_{i, \alpha }(Y_i, \boldsymbol{\theta })}{\partial \theta _{j}}>n\right\} }(Y_i)\right] \right) & =0, \\
\lim_{n\rightarrow \infty }\left( \sup_{n>1}\frac{1}{n}\tsum \limits_{i=1}^{n}E_{Y_i}\left[ \left\vert \frac{\partial ^{2}\widehat{V}_{i, \alpha }(Y_i, \boldsymbol{\theta })}{\partial \theta _{j}\partial \theta_{k}}\right\vert {\LARGE I}_{\left\{ \frac{\partial ^{2}V_{i, \alpha }(Y_i, \boldsymbol{\theta })}{\partial \theta _{j}\partial \theta _{k}}>n\right\} }(Y_i)\right] \right) & =0, \\
\lim_{n\rightarrow \infty }\left( \sup_{n>1}\frac{1}{n}\tsum \limits_{i=1}^{n}E_{Y_i}\left[ \left\vert \frac{\partial ^{3}\widehat{V}_{i, \alpha }(Y_i, \boldsymbol{\theta })}{\partial \theta _{j}\partial \theta _{k}\partial \theta_{l}}\right\vert {\LARGE I}_{\left\{ \frac{\partial ^{3}V_{i, \alpha }(Y_i, \boldsymbol{\theta })}{\partial \theta _{j}\partial \theta _{k}\partial \theta _{l}}>n\right\} }(Y_i)\right] \right) & =0.
\end{align*}

\item[C7.\label{itm:C7 copy(1)}] For all $\varepsilon >0$

\begin{equation*}
\lim_{n\rightarrow \infty }\left\{ \frac{1}{n}\tsum\limits_{i=1}^{n}E_{Y_i}\left[ \left\Vert \boldsymbol{\Omega }_{n}^{-\frac{1}{2}}\left( \boldsymbol{\theta }\right) \frac{\partial \widehat{V}_{i, \alpha }(Y_i,\boldsymbol{\theta } )}{\partial \boldsymbol{\theta }}\right\Vert _{2}^{2}{\LARGE I}_{\left\{ \left\Vert \boldsymbol{\Omega }_{n}^{-\frac{1}{2}}\left( \boldsymbol{\theta }\right) \frac{\partial \widehat{V}_{i, \alpha }(Y_i,\theta )}{\partial \boldsymbol{\theta }}\right\Vert _{2}^{2}\right\} }(Y_i)\right] >\varepsilon \sqrt{n}\right\} =0.
\end{equation*}
\end{description}

Now, the following result, whose proof can be seen in \cite{cajapa22}, holds.

\begin{theorem}
Suppose the previous regularity conditions {\bf C1- C7} hold. Then,

\begin{equation}\label{dist-asymp}
\sqrt{n}\boldsymbol{\Omega }_{n, \alpha }\left( \boldsymbol{\theta }_{\boldsymbol{g}, \alpha }\right)^{-\frac{1}{2}}\boldsymbol{\Psi }_{n, \alpha }\left( \boldsymbol{\theta }_{\boldsymbol{g}, \alpha }\right) \left(
\widehat{\boldsymbol{\theta }}_{\alpha }-\boldsymbol{\theta }_{\boldsymbol{g}, \alpha }\right) \underset{n\rightarrow \infty }{\overset{L}{\rightarrow }}N(\boldsymbol{0}_{p},\boldsymbol{I}_{p}),
\end{equation}
being $\boldsymbol{I}_{p}$ the p-dimensional identity matrix.
\end{theorem}

\subsection{Example: The MPRE under the MLRM \label{sec:MLRM}}

Consider $(Y_{1},..., Y_{n})$ a set of random variables, related to the explanatory variables $(\boldsymbol{X}_{1},..., \boldsymbol{X}_{n})$ through the  MLRM,
\begin{equation}
Y_{i}=\boldsymbol{X}_{i}^{T}\boldsymbol{\beta }+\varepsilon _{i},\quad i=1,\dots ,n,  \label{eq:linear}
\end{equation}
where the errors $\varepsilon _{i}^{\prime }s$ are i.i.d. normal random variables with mean zero and variance $\sigma ^{2}$, $\boldsymbol{X}_{i}^{T}=(X_{i1},...,X_{ip})$ is the vector of independent variables
corresponding to the $i$-th condition and $\boldsymbol{\beta }=\left( \beta_{1},...,\beta _{p}\right) ^{T}$ is the vector of regression coefficients to be estimated. We will consider that, for each $i$, $\boldsymbol{X}_{i}$ is fixed, yielding to i.n.i.d.o. $Y_{i}^{\prime }s$, with $Y_{i}\sim \mathcal{N}(\boldsymbol{X}_{i}^{T}\boldsymbol{\beta },\sigma ^{2})$.

We next derive the explicit expression of the MRPE  for the parameters $\boldsymbol{\theta} = (\boldsymbol{\beta}, \sigma)$. With the previous notation, the assumed density functions are $f_{i}\left( y,\boldsymbol{\beta },\sigma \right) \equiv \mathcal{N}(\boldsymbol{X}_{i}^{T}\boldsymbol{\beta },\sigma ^{2})$ and then, using Eq. (\ref{Aux}), we have that for $\alpha >0$,
\begin{equation} \label{Vi}
	\begin{aligned}
		\widehat{V}_{i, \alpha }(Y_{i};\boldsymbol{\beta },\sigma )&=
		- \frac{\frac{1}{(2\pi )^{\alpha /2}\sigma ^{\alpha }}\exp \left( \frac{-\alpha (Y_{i}-\boldsymbol{X}_{i}^{T}\boldsymbol{\beta })^{2}}{2\sigma ^{2}}\right) }{\alpha \left( (2\pi )^{\alpha /2}\sigma ^{\alpha }\sqrt{1+\alpha }\right) ^{-\frac{\alpha }{\alpha +1}}}  + {1\over \alpha }\\
		&=
		- {1\over \alpha }\left( \frac{1+\alpha }{2\pi }\right) ^{\frac{\alpha }{2(\alpha +1)}}\sigma ^{-\frac{\alpha }{\alpha +1}}\exp \left( -\frac{\alpha }{2}\left( \frac{Y_{i}-\boldsymbol{X}_{i}^{T}\boldsymbol{\beta }}{\sigma }\right)^{2}\right) + {1\over \alpha }.
	\end{aligned}
\end{equation}
and thus,  the MRPE for $\alpha >0$ is obtained minimizing the averaged objective function
\begin{equation*}
	\begin{aligned}
		H_{n,\alpha} (\boldsymbol{\theta}) &= \frac{1}{n}\sum_{i=1}^{n}\widehat{V}_{i, \alpha }(Y_{i};\boldsymbol{\beta },\sigma )\\
		&=
		- {1\over \alpha }\left(\frac{1+\alpha }{2\pi }\right) ^{\frac{\alpha }{2(\alpha +1)}}\frac{1}{n}\sum_{i=1}^{n}\sigma ^{-\frac{\alpha }{\alpha +1}}\exp \left( -\frac{\alpha }{2}\left( \frac{Y_{i}-\boldsymbol{X}_{i}^{T}\boldsymbol{\beta }}{\sigma }\right) ^{2}\right)  + {1\over \alpha }.
	\end{aligned}
\end{equation*}

Ignoring all constant terms, we have that the MRPE for the MLRM is given, for $\alpha >0,$ as
\begin{equation*}
\left( \widehat{\boldsymbol{\beta }}_{\alpha },\widehat{\sigma }_{\alpha }\right) =
\arg \min_{\boldsymbol{\beta },\sigma }{\sum\limits_{i=1}^{n}}- \sigma ^{-\frac{\alpha }{\alpha +1}}\exp \left( -\frac{\alpha }{2}\left( \frac{Y_{i}-\boldsymbol{X}_{i}^{T}\boldsymbol{\beta }}{\sigma }\right)^{2}\right) .
\end{equation*}

Moreover, taking derivatives with respect to $\boldsymbol{\beta }$ and $\sigma,$ the estimation equations of $\widehat{\boldsymbol{\beta }}_{\alpha }$ and $\widehat{\sigma }_{\alpha }$ are
\begin{equation}
\begin{array}{l}
{\sum\limits_{i=1}^{n}}\exp \left( -\frac{\alpha }{2}\left( \frac{Y_{i}-\boldsymbol{X}_{i}^{T}\boldsymbol{\beta }}{\sigma }\right) ^{2}\right) \left( \frac{Y_{i}-\boldsymbol{X}_{i}^{T}\boldsymbol{\beta }}{\sigma }\right) \boldsymbol{X}_{i}=\boldsymbol{0}_{p} \\
{\sum\limits_{i=1}^{n}}\exp \left( -\frac{\alpha }{2}\left( \frac{Y_{i}-\boldsymbol{X}_{i}^{T}\boldsymbol{\beta }}{\sigma }\right) ^{2}\right) \left\{ \left( \frac{Y_{i}-\boldsymbol{X}_{i}^{T}\boldsymbol{\beta }}{\sigma
}\right) ^{2}-\frac{1}{1+\alpha }\right\} =0
\end{array}
,  \label{eq:estimating}
\end{equation}
which is exactly the same system as the one obtained in \cite{cajapa22}. For $\alpha =0$, if we denote $\mathbb{X}= (\boldsymbol{X}_1, ..., \boldsymbol{X}_n)_{n\times p}^{T}$ and $\boldsymbol{Y} =(Y_1, ..., Y_n),$ we get the MLE of $\widehat{\boldsymbol{\beta }}_0$ and $\widehat{\sigma }_0,$ i.e.

\begin{equation*}
\widehat{\boldsymbol{\beta }}_{0}=(\mathbb{X}^{T}\mathbb{X}\mathbf{)}^{-1}\mathbb{X}^{T}\mathbf{Y}\text{ \ and \ }\widehat{\sigma }_{0}^{2}=\frac{1}{n}{\sum\limits_{i=1}^{n}}\left( Y_{i}-\boldsymbol{X}_{i}^{T}\widehat{\boldsymbol{\beta }}_{0}\right) ^{2}.
\end{equation*}

Finally, from the results in \cite{cajapa22}, it can be seen that matrices $\boldsymbol{\Psi }_{n,\alpha}\left( \boldsymbol{\beta },\sigma \right) $ and $\boldsymbol{\Omega }_{n,\alpha}\left( \boldsymbol{\beta },\sigma \right) $ are given by
\begin{eqnarray*}
\boldsymbol{\Psi }_{n,\alpha}\left( \boldsymbol{\beta },\sigma \right)  & = & \frac{1}{n}\tsum\limits_{i=1}^{n}\boldsymbol{J}^{\left( i\right) }\left( \boldsymbol{\beta },\sigma ^{2}\right)  \\
& = & k\sigma ^{-\frac{3\alpha +2}{\alpha +1}} \left( \alpha +1\right) ^{-\frac{3}{2}}\left[ \begin{array}{cc} \frac{1}{n}\mathbb{X}^{T}\mathbb{X} & 0 \\ 0 & \frac{2}{\alpha +1} \end{array} \right] \\
& = & K_{1}\left( \alpha +1\right) ^{-\frac{3}{2}}\left[ \begin{array}{cc} \frac{1}{n}\mathbb{X}^{T}\mathbb{X} & 0 \\ 0 & \frac{2}{\alpha +1} \end{array} \right] ,
\end{eqnarray*}
and
\begin{eqnarray*}
	\boldsymbol{\Omega }_{n,\alpha}\left( \boldsymbol{\beta },\sigma \right)
	& = & \frac{1}{n}\tsum\limits_{i=1}^{n}Var_{Y_i}\left[ \left( \frac{\partial V_{i, \alpha }(Y_i;\boldsymbol{\beta },\sigma ^{2})}{\partial \theta _{j}}\right) _{j=1,..,k}\right] \\
	& = & K_{1}^{2}\sigma ^{2}\frac{1}{\left( 2\alpha +1\right) ^{3/2}}\left[
	\begin{array}{cc} \frac{1}{n}\mathbb{X}^{T}\mathbb{X} & \boldsymbol{0} \\ \boldsymbol{0} & \frac{(3\alpha ^{2}+4\alpha +2)}{(\alpha +1)^{2}(2\alpha +1)}\end{array}
	\right] .
\end{eqnarray*}
with
\begin{equation}\label{K1}
k= {1\over \alpha} \left( {1+\alpha \over 2\pi }\right)^{\alpha \over 2(\alpha +1)} ,\quad K_{1}=k\sigma ^{-\frac{3\alpha +2}{\alpha +1}}.
\end{equation}

Therefore, for $\alpha =0$ we get the Fisher information matrix for $\left( \boldsymbol{\beta },\sigma \right) $ in both matrices, i.e.
$$ \boldsymbol{\Psi }_{n,0}\left( \boldsymbol{\beta },\sigma \right) = \left[ \begin{array}{cc} {1\over \sigma^2}\frac{1}{n}\mathbb{X}^{T}\mathbb{X} & 0 \\ 0 & \frac{2}{\sigma^2} \end{array} \right],$$
and
$$ \boldsymbol{\Omega }_{n,0}\left( \boldsymbol{\beta },\sigma \right) = \left[ \begin{array}{cc} {1\over \sigma^2}\frac{1}{n}\mathbb{X}^{T}\mathbb{X} & \boldsymbol{0} \\ \boldsymbol{0} & \frac{2}{\sigma^2}\end{array}\right] .$$

\section{Model selection criterion based on RP}

In this section we present the model selection criterion based on RP. Let us consider a collection of $l$ candidate models
\begin{equation}
\left\{ \boldsymbol{M}^{(s)} = \left( M_{1}^{(s)},...,M_{n}^{(s)}\right) \right\} _{s\in \left\{ 1,...,l\right\} }  \label{6.4.1}
\end{equation}
  such that each $\boldsymbol{M}^{(s)} $ is characterized by the parametric density functions
\begin{equation*}
\boldsymbol{f}(\cdot,\boldsymbol{\theta }_{s}) = \left( f_{1}(\cdot,\boldsymbol{\theta }_{s}),...,f_{n}(\cdot,\boldsymbol{\theta }_{s})\right) ,\text{ }\boldsymbol{\theta }_{s}\in \boldsymbol{\Theta }_{s}\subset \mathbb{R}^{p_s},
\end{equation*}
with associated distribution functions 
 $\boldsymbol{F}(., \boldsymbol{\theta }_s)=\left( F_1(\boldsymbol{\theta }_s),...,F_n(., \boldsymbol{\theta }_s)\right) ,$ where $\boldsymbol{\theta }_s$ is common for all density functions in model $s.$ That is, each candidate model would represent a parametric family defined by a common parameter, which may contain different number of parameters.
 Based on the random sample $Y_{1},...,Y_{n},$ we need to select the best model from the collection $\{ \boldsymbol{M}^{(s)} \} _{s\in \left\{ 1,...,l\right\} }$ according to some suitable selection criterion. For such purpose, for each assumed model $\boldsymbol{M}^{(s)},$ we should first determine the best parameter $\boldsymbol{\theta}_s$ fitting the sample and subsequently select the best fitted model from the collection. Then, given a set of observations, the model selection is performed in two steps: we first fit all the candidates models to the data, and then select the model with best trade-off between goodness of fit and complexity in terms of RP.

We next describe the first step of the model selection algorithm.
Let consider a fixed parametric model $\boldsymbol{M}^{(s)}$ modeling the true distribution underlying.
If the true distribution was known, the parameter that best fits the model $\boldsymbol{M}^{(s)},$  denoted by $\boldsymbol{\theta }_{\boldsymbol{g}, \alpha }^s,$ can be obtained by maximizing the theoretical averaged objective function  $H_{\alpha }(\boldsymbol{\theta })$ defined in Eq. (\ref{A}) under the $s$-model.
%

Following the discussion in Section 2, if the true distribution underlying is unknown but we have a random sample $Y_1,...,Y_n,$ 
the best estimate of the true parameter based on the sample from the RP approach is the MRPE defined in (\ref{2.4.1}).

Once all candidate models are fitted to the observed data (or to the true distribution, if it is known), we should select the model with the best trade-off between fitness and complexity. Therefore, we need a measure of fairness between the best candidate for each model and the true distribution.
 The goodness of fit of a certain model $\boldsymbol{M}^{(s)}$ with associated densities $\boldsymbol{f}(\cdot, \boldsymbol{\theta}^s_g)$ and the best-fitting parameter $\boldsymbol{\theta}^s_g$ based on the RP can be quantified by the averaged objective function $H_{\alpha}(\boldsymbol{\theta}^s_g)$ given in Eq. (\ref{A}).

As the true distribution is generally unknown, $\boldsymbol{\theta}^s_g$ is estimated by $\widehat{\boldsymbol{\theta }}_{\alpha }^s$. Hence, we can estimate $H_{\alpha}(\boldsymbol{\theta}^s_g)$ by $H_{\alpha}(\widehat{\boldsymbol{\theta }}_{\alpha }^s).$ But again $H_{\alpha }$ needs to be estimated, and the natural estimator is $H_{n, \alpha }( \widehat{\boldsymbol{\theta }}_{\alpha }^s).$ However, as the sample is used both for estimating the parameter and for estimating $H_{\alpha },$ it does not hold that

$$E_{Y_{1},...,Y_{n}}\left[ H_{n, \alpha }\left( \widehat{\boldsymbol{\theta }}_{\alpha }^s\right) \right] \ne E_{Y_{1},...,Y_{n}}\left[ H_{\alpha }\left( \widehat{\boldsymbol{\theta }}_{\alpha }^s \right) \right] .$$

Moreover, the estimation bias would depend on the model and consequently, we need to add a term correcting the bias caused by the model assumption.

The AIC criterion selects the model that minimizes
$$ - 2\sum_{i=1}^n \log f_i(y_i, \boldsymbol{\theta }) + 2 p = 2H_{n,0}\left( \boldsymbol{\theta }\right) + 2 p,  $$
 where $2p$ is the term correcting the bias. Following the same idea, we define the   $RP_{NH}-$Criterion as follows:

\begin{definition}\label{def-Renyi-criterion}
Let $\left\{ \left( M_{1}^{(s)},...,M_{n}^{(s)}\right) \right\} _{s\in \left\{ 1,..., l\right\} }$ be $l$ candidate models for the i.n.i.d.o. $Y_{1},...,Y_{n}$. The selected model $\left( M_{1}^{\ast },...,M_{n}^{\ast
}\right) $ according the {\bf $RP_{NH}-$Criterion} is the one satisfying

\begin{equation*}
\left( M_{1}^{\ast },...,M_{n}^{\ast }\right) =\min_{s\in \{ 1, ..., l\} }RP_{NH}\left( M_{1}^{(s)},...,M_{n}^{(s)},\widehat{\boldsymbol{\theta }}_{\alpha }^s\right) ,
\end{equation*}
where

\begin{equation}
RP_{NH}\left( M_{1}^{(s)},...,M_{n}^{(s)},\widehat{\boldsymbol{\theta }}_{\alpha }^s\right) =  H_{n, \alpha }\left( \widehat{\boldsymbol{\theta }}_{\alpha }^s\right) +\frac{1}{n}trace\left( \boldsymbol{\boldsymbol{\Omega }_{n}\left(
\widehat{\boldsymbol{\theta }}_{\alpha }^s\right) \Psi }_{n}^{-1}\left( \widehat{\boldsymbol{\theta }}_{\alpha }^s\right) \right) .  \label{CRP}
\end{equation}
\end{definition}

We can observe that

$$ \lim_{\alpha \rightarrow 0}RP_{NH}\left( M_{1}^{(s)},...,M_{n}^{(s)},\widehat{\boldsymbol{\theta }}_{\alpha }^s\right) = -{1\over n}\sum_{i=1}^n \log f_i(Y_i, \boldsymbol{\theta }) + {p\over n},$$
and hence we recover AIC criterion up to the multiplicative constant $2n.$

In order to justify the $RP_{NH}-$Criterion, 
we shall establish that the estimated function $RP_{NH}\left( M_{1}^{(s)},...,M_{n}^{(s)}\right) $ quantifying the loss of choosing a model is an unbiased estimator of it theoretical version,
$E_{Y_{1},...,Y_{n}}\left[ H_{\alpha }\left( \widehat{\boldsymbol{\theta }}_{\alpha }^s\right) \right] .$
For this purpose, we shall assume the following additional regularity condition:
\begin{description}
\item[C8.\label{itm:C8}] The matrices $\boldsymbol{\Psi }_{n}^{-1}\left( \boldsymbol{\theta }\right) $ and $\boldsymbol{\Omega }_{n}\left( \boldsymbol{\theta }\right) $ are continuous for arbitrary $\boldsymbol{\theta \in \Theta }.$
\end{description}

\begin{theorem}
Assume that conditions {\bf C1-C8} hold. Then,
\begin{equation*}
E_{Y_{1},...,Y_{n}}\left[ RP_{NH}\left( M_{1}^{(s)},...,M_{n}^{(s)}, \widehat{\boldsymbol{\theta }}_{\alpha }^s \right) \right] =E_{Y_{1},...,Y_{n}}\left[ H_{\alpha }\left( \widehat{\boldsymbol{\theta }}_{\alpha }^s \right) \right] ,\forall s=1, ..., l.
\end{equation*}
\end{theorem}

\begin{proof}
Consider a fixed $s=1, ..., l.$ A Taylor expansion of $V_{i, \alpha }\left( \boldsymbol{\theta }\right) $ defined in Eq. (\ref{Aux}) around $\boldsymbol{\theta }_{\boldsymbol{g}, \alpha}^s$ and evaluated at $\widehat{\boldsymbol{\theta }}_{\alpha }^s$ gives

\begin{eqnarray*}
V_{i, \alpha }\left( \widehat{\boldsymbol{\theta }}_{\alpha }^s\right)
& = & V_{i, \alpha }\left( \boldsymbol{\theta }_{\boldsymbol{g}, \alpha }^s\right) +\left( \frac{\partial V_{i, \alpha }\left( \boldsymbol{\theta }\right) }{\partial \boldsymbol{\theta }}\right) _{\boldsymbol{\theta =\theta }_{\boldsymbol{g}, \alpha}^s}\left( \widehat{\boldsymbol{\theta }}_{\alpha }^s-\boldsymbol{\theta }_{\boldsymbol{g}, \alpha }^s\right)  \\
& & +\frac{1}{2}\left( \widehat{\boldsymbol{\theta }}_{\alpha }^s-\boldsymbol{\theta }_{\boldsymbol{g}, \alpha }^s\right) ^{T}\left( \frac{\partial ^{2}V_{i, \alpha }\left( \boldsymbol{\theta }\right) }{\partial \boldsymbol{\theta }\text{ }\partial \boldsymbol{\theta }^{T}}\right) _{\boldsymbol{\theta =\theta }_{\boldsymbol{g}, \alpha }^s}\left( \widehat{\boldsymbol{\theta }}_{\alpha }^s-\boldsymbol{\theta }_{\boldsymbol{g}, \alpha }^s\right) +o\left( \left\Vert \widehat{\boldsymbol{\theta }}_{\alpha }^s-\boldsymbol{\theta }_{\boldsymbol{g}, \alpha }^s\right\Vert ^{2}\right)  \\
& = & V_{i, \alpha }\left( \boldsymbol{\theta }_{\boldsymbol{g}, \alpha }^s\right) +\left( \frac{\partial V_{i, \alpha }\left( \boldsymbol{\theta }\right) }{\partial \boldsymbol{\theta }}\right) _{\boldsymbol{\theta =\theta }_{\boldsymbol{g}, \alpha }^s}\left( \widehat{\boldsymbol{\theta }}_{\alpha }^s-\boldsymbol{\theta }_{\boldsymbol{g}, \alpha }^s\right)  \\
& & +\frac{1}{2}\left( \widehat{\boldsymbol{\theta }}_{\alpha }^s-\boldsymbol{\theta }_{\boldsymbol{g}, \alpha }^s\right) ^{T} \boldsymbol{J}_{\tau }^{(i)}\left( \boldsymbol{\theta }_{\boldsymbol{g}, \alpha }^s\right) \left(
\widehat{\boldsymbol{\theta }}_{\alpha }^s-\boldsymbol{\theta }_{\boldsymbol{g}, \alpha }^s\right) + o\left( \left\Vert \widehat{\boldsymbol{\theta }}_{\alpha }^s-\boldsymbol{\theta }_{\boldsymbol{g}, \alpha }^s\right\Vert ^{2}\right) .
\end{eqnarray*}

Summing over $i$ and dividing by $n$, taking into account that $\boldsymbol{\theta }_{\boldsymbol{g}, \alpha }^s$ maximizes $H_{\alpha }(\boldsymbol{\theta}),$ we get

\begin{equation*}
H_{\alpha }( \widehat{\boldsymbol{\theta }}_{\alpha }^s) = H_{\alpha }\left( \boldsymbol{\theta }_{\boldsymbol{g}, \alpha }^s\right) - \frac{1}{2}\left( \widehat{\boldsymbol{\theta }}_{\alpha }^s-\boldsymbol{\theta }_{\boldsymbol{g}, \alpha }^s\right) ^{T}\boldsymbol{\Psi }_{n}\left( \boldsymbol{\theta }_{\boldsymbol{g}, \alpha }^s\right) \left( \widehat{\boldsymbol{\theta }}_{\alpha }^s-\boldsymbol{\theta }_{\boldsymbol{g}, \alpha }^s\right) +o\left( \left\Vert \widehat{\boldsymbol{\theta }}_{\alpha }^s-\boldsymbol{\theta }_{\boldsymbol{g}, \alpha }^s\right\Vert^{2}\right)
\end{equation*}
and hence,

\begin{equation}\label{eq-aux}
E_{Y_{1},...,Y_{n}}\left[ nH_{\alpha }\left( \widehat{\boldsymbol{\theta }}_{\alpha }^s\right) \right] = n H_{\alpha }\left( \boldsymbol{\theta }_{\boldsymbol{g}, \alpha }^s\right) - \frac{1}{2}E_{Y_{1},...,Y_{n}}\left[ \sqrt{n}\left( \widehat{\boldsymbol{\theta }}_{\alpha }^s-\boldsymbol{\theta }_{\boldsymbol{g}, \alpha }^s\right) ^{T}\boldsymbol{\Psi }_{n}\left( \boldsymbol{\theta }_{\boldsymbol{g}, \alpha }^s\right) \sqrt{n}\left( \widehat{\boldsymbol{\theta }}_{\alpha }^s-\boldsymbol{\theta }_{\boldsymbol{g}, \alpha }^s\right) \right] +o_{p}(1).
\end{equation}

But by Eq. (\ref{dist-asymp}), and applying Corollary 2.1 in \cite{digu85}, we have

\begin{equation*}
\sqrt{n}\left( \widehat{\boldsymbol{\theta }}_{\alpha }^s-\boldsymbol{\theta }_{\boldsymbol{g}, \alpha }^s\right) ^{T}\boldsymbol{\Psi }_{n}\left( \boldsymbol{\theta }_{\boldsymbol{g}, \alpha }^s\right) \sqrt{n}\left( \widehat{\boldsymbol{\theta }}_{\alpha }^s-\boldsymbol{\theta }_{\boldsymbol{g}, \alpha }^s\right) \underset{n\longrightarrow \infty }{\overset{\mathcal{L}}{\longrightarrow }}\dsum_{i=1}^{k}\lambda _{i}(\boldsymbol{\theta }_{\boldsymbol{g}, \alpha }^s)Z_{i}^{2},
\end{equation*}
where $\lambda _{1}(\boldsymbol{\theta }_{\boldsymbol{g}, \alpha }^s),...,\lambda _{n}(\boldsymbol{\theta }_{\boldsymbol{g}, \alpha }^s)$ are the eigenvalues of the matrix
\begin{equation*}
\boldsymbol{\Psi }_{n}\left( \boldsymbol{\theta }_{\boldsymbol{g}, \alpha }^s\right) \boldsymbol{\Psi }_{n}\left( \boldsymbol{\theta }_{\boldsymbol{g}, \alpha }^s\right) ^{-1}\boldsymbol{\Omega }_{n}\left( \boldsymbol{\theta }_{\boldsymbol{g}, \alpha }^s\right) \boldsymbol{\Psi }_{n}\left( \boldsymbol{\theta }_{\boldsymbol{g}, \alpha }^s\right) ^{-1}=\boldsymbol{\Omega }_{n}\left( \boldsymbol{\theta }_{\boldsymbol{g}, \alpha }^s\right) \boldsymbol{\Psi }_{n}\left( \boldsymbol{\theta }_{\boldsymbol{g}, \alpha }^s\right) ^{-1}
\end{equation*}
$^{{}}$and $Z_{1},...,Z_{k}$ are independent normal random variables with mean zero and variance 1. Therefore,

\begin{eqnarray*}
E_{Y_{1},...,Y_{n}} &\left[ \sqrt{n}\left( \widehat{\boldsymbol{\theta }}_{\alpha }^s-\boldsymbol{\theta }_{\boldsymbol{g}, \alpha }^s\right) ^{T}\boldsymbol{\Psi }_{n}\left( \boldsymbol{\theta }_{\boldsymbol{g}, \alpha }^s\right) \sqrt{n}\left( \widehat{\boldsymbol{\theta }}_{\alpha }^s-\boldsymbol{\theta }_{\boldsymbol{g}, \alpha }^s\right) \right]\\
 = & \dsum_{i=1}^{k}\lambda _{i}(\boldsymbol{\theta }_{\boldsymbol{g}, \alpha }^s)+o_{P}(1) \\
 = & trace\left( \boldsymbol{\Omega }_{n}\left( \boldsymbol{\theta }_{\boldsymbol{g}, \alpha }^s\right) \boldsymbol{\Psi }_{n}\left( \boldsymbol{\theta }_{\boldsymbol{g}, \alpha }^s\right)^{-1}\right) +o_{P}(1).
\end{eqnarray*}

On the other hand, taking into account that $\widehat{\boldsymbol{\theta }}_{\alpha }^s$ maximizes $H_{n, \alpha }\left( \boldsymbol{\theta } \right) ,$ a Taylor expansion of $H_{n, \alpha }\left( \boldsymbol{\theta }\right) $ at $\widehat{\boldsymbol{\theta }}_{\alpha }^s$ and evaluated at $\boldsymbol{\theta }_{\boldsymbol{g}, \alpha }^s$ gives

\begin{equation*}
H_{n,\alpha }\left( \boldsymbol{\theta }_{\boldsymbol{g}, \alpha }^s\right) = H_{n,\alpha }\left( \widehat{\boldsymbol{\theta }}_{\alpha }^s\right) +\frac{1}{2}\left( \boldsymbol{\theta }_{\boldsymbol{g}, \alpha }^s-\widehat{\boldsymbol{\theta }}_{\alpha }^s \right)^{T}\left( \frac{\partial ^{2}H_{n,\alpha }\left( \boldsymbol{\theta }\right) }{\partial \boldsymbol{\theta }\text{ }\partial \boldsymbol{\theta }^{T}}\right) _{\boldsymbol{\theta =}\widehat{\boldsymbol{\theta }}_{\alpha }^s}\left( \boldsymbol{\theta }_{\boldsymbol{g}, \alpha }^s-\widehat{\boldsymbol{\theta }}_{\alpha }^s\right) +o\left( \left\Vert \boldsymbol{\theta }_{\boldsymbol{g}, \alpha }^s-\widehat{\boldsymbol{\theta }}_{\alpha }^s\right\Vert ^{2}\right) .
\end{equation*}

But then, multiplying by $n$ and considering the expected values,

\begin{eqnarray*}
n H_{\alpha }\left( \boldsymbol{\theta }_{\boldsymbol{g}, \alpha }^s\right)
& = & E_{Y_{1},...,Y_{n}}\left[ nH_{n,\alpha }\left( \boldsymbol{\theta }_{\boldsymbol{g}, \alpha }^s\right) \right] = E_{Y_{1},...,Y_{n}}\left[ nH_{n,\alpha }\left( \widehat{\boldsymbol{\theta }}_{\alpha }^s\right) \right]  \\
& & +\frac{1}{2}E_{Y_{1},...,Y_{n}}\left[ \sqrt{n}\left( \boldsymbol{\theta }_{\boldsymbol{g}, \alpha }^s-\widehat{\boldsymbol{\theta }}_{\alpha }^s\right) ^{T}\left( \frac{\partial ^{2}H_{n,\alpha }\left( \boldsymbol{\theta }\right) }{\partial \boldsymbol{\theta }\text{ }\partial \boldsymbol{\theta }^{T}}\right) _{\boldsymbol{\theta =}\widehat{\boldsymbol{\theta }}_{\alpha }^s}\sqrt{n}\left( \boldsymbol{\theta }_{\boldsymbol{g}, \alpha }^s-\widehat{\boldsymbol{\theta }}_{\alpha }^s\right) \right] +o_{p}(1).
\end{eqnarray*}

Besides,

\begin{equation}\label{Ecuacion}
\left( \frac{\partial ^{2}H_{n,\alpha }\left( \boldsymbol{\theta }\right) }{\partial \boldsymbol{\theta }\text{ }\partial \boldsymbol{\theta }^{T}}\right) _{\boldsymbol{\theta =}\widehat{\boldsymbol{\theta }}_{\alpha }^s}\underset{n\longrightarrow \infty }{\overset{\mathcal{P}}{\longrightarrow }} - \boldsymbol{\Psi }_{n}\left( \boldsymbol{\theta }_{\boldsymbol{g}, \alpha }^s\right) .
\end{equation}
by the continuity of $\boldsymbol{\Psi }_{n}.$ Hence, substituting in (\ref{eq-aux})

\small{
\begin{align*}
 E_{Y_{1},...,Y_{n}}&\left[ nH_{\alpha }\left( \widehat{\boldsymbol{\theta }}_{\alpha }^s\right) \right] \\
  = & n H_{\alpha }\left( \boldsymbol{\theta }_{\boldsymbol{g}, \alpha }^s\right) - \frac{1}{2}E_{Y_{1},...,Y_{n}}\left[ \sqrt{n}\left( \widehat{\boldsymbol{\theta }}_{\alpha }^s-\boldsymbol{\theta }_{\boldsymbol{g}, \alpha }^s\right) ^{T}\boldsymbol{\Psi }_{n}\left( \boldsymbol{\theta }_{\boldsymbol{g}, \alpha }^s\right) \sqrt{n}\left( \widehat{\boldsymbol{\theta }}_{\alpha }^s-\boldsymbol{\theta }_{\boldsymbol{g}, \alpha }^s\right) \right] +o_{p}(1) \\
 = & E_{Y_{1},...,Y_{n}}\left[ nH_{n,\alpha }\left( \widehat{\boldsymbol{\theta }}_{\alpha }^s\right) \right] - \frac{1}{2}E_{Y_{1},...,Y_{n}}\left[ \sqrt{n}\left( \widehat{\boldsymbol{\theta }}_{\alpha }^s-\boldsymbol{\theta }_{\boldsymbol{g}, \alpha }^s\right) ^{T}\boldsymbol{\Psi }_{n}\left( \boldsymbol{\theta }_{\boldsymbol{g}, \alpha }^s\right) \sqrt{n}\left( \widehat{\boldsymbol{\theta }}_{\alpha }^s-\boldsymbol{\theta }_{\boldsymbol{g}, \alpha }^s\right) \right] +o_{p}(1) \\
 & - \frac{1}{2}E_{Y_{1},...,Y_{n}}\left[ \sqrt{n}\left( \boldsymbol{\theta }_{\boldsymbol{g}, \alpha }^s-\widehat{\boldsymbol{\theta }}_{\alpha }^s\right) ^{T}\boldsymbol{\Psi }_{n}\left( \boldsymbol{\theta }_{\boldsymbol{g}, \alpha }^s\right) \sqrt{n}\left( \boldsymbol{\theta }_{\boldsymbol{g}, \alpha }^s-\widehat{\boldsymbol{\theta }}_{\alpha }^s\right) \right] +o_{p}(1) \\
 = & E_{Y_{1},...,Y_{n}}\left[ nH_{n,\alpha }\left( \widehat{\boldsymbol{\theta }}_{\alpha }^s\right) \right] - E_{Y_{1},...,Y_{n}}\left[ \sqrt{n}\left( \widehat{\boldsymbol{\theta }}_{\alpha }^s-\boldsymbol{\theta }_{\boldsymbol{g}, \alpha }^s\right) ^{T}\boldsymbol{\Psi }_{n}\left( \boldsymbol{\theta }_{\boldsymbol{g}, \alpha }^s\right) \sqrt{n}\left( \widehat{\boldsymbol{\theta }}_{\alpha }^s-\boldsymbol{\theta }_{\boldsymbol{g}, \alpha }^s\right) \right] +o_{p}(1),
\end{align*}}
and thus,

\begin{eqnarray*}
E_{Y_{1},...,Y_{n}}\left[ H_{\alpha }\left( \widehat{\boldsymbol{\theta }}_{\alpha }^s\right) \right] & = & E_{Y_{1},...,Y_{n}}\left[ H_{n,\alpha }\left( \widehat{\boldsymbol{\theta }}_{\alpha }^s\right) \right]  \\
& & - \frac{1}{n}E_{Y_{1},...,Y_{n}}\left[ \sqrt{n}\left( \boldsymbol{\theta }_{\boldsymbol{g}, \alpha }^s-\widehat{\boldsymbol{\theta }}_{\alpha }^s\right) ^{T}\boldsymbol{\Psi }_{n}\left( \boldsymbol{\theta }_{\boldsymbol{g}, \alpha }^s\right) \sqrt{n}\left( \boldsymbol{\theta }_{\boldsymbol{g}, \alpha }^s-\widehat{\boldsymbol{\theta }}_{\alpha }^s\right) \right] +o_{p}(1) \\
& = & E_{Y_{1},...,Y_{n}}\left[ H_{n, \alpha }\left( \widehat{\boldsymbol{\theta }}_{\alpha }^s\right) \right] - \frac{1}{n}trace\left( \boldsymbol{\Omega }_{n}\left( \boldsymbol{\theta }_{\boldsymbol{g}, \alpha }^s\right) \boldsymbol{\Psi }_{n}^{-1}\left( \boldsymbol{\theta }_{\boldsymbol{g}, \alpha }^s\right) \right).
\end{eqnarray*}

Hence, the result holds.
\end{proof}

We next develop explicit expressions for the $RP_{NH}$-criterion under the MLRM.

\subsection{Example: The RP-based model selection under the multiple linear regression model \label{sec:RPMLRM}}

We consider the MLRM defined in Section 2.1.

\begin{equation}
Y_{i}=\boldsymbol{X}_{i}^{T}\boldsymbol{\beta }+\varepsilon _{i},\quad i=1,\dots ,n.  \label{eq:linear}
\end{equation}

We consider several models $\{ (M^{(s)}_1, ..., M^{(s)}_n)\}_{s=1, ..., l}$ where each model differs on the parameter $\boldsymbol{\beta }$ considered. For example, consider  four explanatory variables $(X_1, X_2, X_3, X_4)$ and four different models given by
\begin{align*}
	 (M^{(1)}_1, ..., M^{(1)}_n) &\equiv Y_i = \beta_0 + \beta_1 X_1 + \beta_2 X_2 + \beta_3 X_3 + \epsilon_i ,\\
	 (M^{(2)}_1, ..., M^{(2)}_n) &\equiv Y_i = \beta_0 + \beta_1 X_1 + \beta_2 X_2 + \beta_4 X_4 + \epsilon_i \\
	(M^{(3)}_1, ..., M^{(3)}_n) &\equiv Y_i = \beta_0 + \beta_1 X_1 + \beta_3 X_3 + \beta_4 X_4 + \epsilon_i ,\\ (M^{(4)}_1, ..., M^{(4)}_n) &\equiv Y_i = \beta_0 + \beta_2 X_2 + \beta_3 X_3 + \beta_4 X_4 + \epsilon_i.
\end{align*}

Each of the models has five parameters that need to be estimated. Let us then determine the corresponding values of $RP_{NH}\left( M_{1}^{(s)},...,M_{n}^{(s)}, \widehat{\boldsymbol{\theta }}_{\alpha }^s \right) $ for $s=1, 2, 3, 4.$

As stated in Section \ref{sec:MLRM}, for each $s=1, 2, 3, 4,$ the estimators of $\widehat{\boldsymbol{\beta }}_{\alpha }^s$ and $\widehat{\sigma }_{\alpha }^s$ are the solutions of the system

\begin{equation}
\begin{array}{l}
{\sum\limits_{i=1}^{n}}\exp \left( -\frac{\alpha }{2}\left( \frac{Y_{i}-\boldsymbol{X}_{s,i}^{T}\boldsymbol{\beta }}{\sigma }\right) ^{2}\right) \left( \frac{Y_{i}-\boldsymbol{X}_{s, i}^{T}\boldsymbol{\beta }}{\sigma }\right) \boldsymbol{X}_{s, i}=\boldsymbol{0}_{4} \\
{\sum\limits_{i=1}^{n}}\exp \left( -\frac{\alpha }{2}\left( \frac{Y_{i}-\boldsymbol{X}_{s, i}^{T}\boldsymbol{\beta }}{\sigma }\right) ^{2}\right) \left\{ \left( \frac{Y_{i}-\boldsymbol{X}_{s, i}^{T}\boldsymbol{\beta }}{\sigma
}\right) ^{2}-\frac{1}{1+\alpha }\right\} =0
\end{array}
,
\end{equation}
where $X_{s, i}$ corresponds to the values of observation $i$ restricted to the variables appearing in model $s.$ Note that, although $\boldsymbol{\beta}$ has a different meaning for the different models, this is not the case of $\sigma.$ However, the  estimation of   $\sigma$ is different for the different models and so this estimation is denoted for  by $\widehat{\sigma }_{\alpha }^s$ for the $s$-th model.

At $\alpha =0$, we have that the model parameters can be explicitly obtained as
\begin{equation*}
\widehat{\boldsymbol{\beta }}_{0}^s=(\mathbb{X}_s^{T}\mathbb{X}_s\mathbf{)}^{-1}\mathbb{X}_s^{T}\mathbf{Y}\text{ \ and \ }(\widehat{\sigma }_{0}^s)^{2}=\frac{1}{n}{\sum\limits_{i=1}^{n}}\left( Y_{i}-\boldsymbol{X}_{s,i}^{T}\widehat{\boldsymbol{\beta }}_{0}\right) ^{2}.
\end{equation*}

Thus, according to Eq. (\ref{2.4}),
$$ H_{n, \alpha }(\widehat{\boldsymbol{\beta}}, \widehat{\sigma }) = {1\over \alpha }{1\over n} \sum_{i=1}^n- k\widehat{\sigma }^{-\frac{\alpha }{\alpha +1}}\exp \left( -\frac{\alpha }{2}\left( \frac{Y_{i}-\boldsymbol{X}_{i}^{T}\widehat{\boldsymbol{\beta }}}{\widehat{\sigma }}\right)^{2}\right) +{1\over \alpha },$$ with $k$ as defined in (\ref{K1}).

Next, let us obtain expressions of $\boldsymbol{\Psi }_{s, n}\left( \boldsymbol{\beta }^s,\sigma \right) $ and $\boldsymbol{\Omega }_{s, n}\left( \boldsymbol{\beta }^s,\sigma \right) .$ Note that these matrices also depend on the model $s$. Applying again the results of the previous section, we obtain

\begin{equation}\label{EqsSR}
\begin{aligned}
	\boldsymbol{\Psi }_{s,n}\left( \boldsymbol{\beta }^s,\sigma \right) &=K_{1}\left( \alpha +1\right) ^{-\frac{3}{2}}\left[
	\begin{array}{cc} \frac{1}{n}\mathbb{X}_s^{T}\mathbb{X}_s & 0 \\ 0 & \frac{2}{\alpha +1} \end{array}
	\right] , \\
	 \boldsymbol{\Omega }_{s, n}\left( \boldsymbol{\beta }^s,\sigma \right) &= K_{1}^{2}\sigma ^{2}\frac{1}{\left( 2\alpha +1\right) ^{3/2}}\left[
	\begin{array}{cc} \frac{1}{n}\mathbb{X}_s^{T}\mathbb{X}_s & 0 \\ 0 & \frac{(3\alpha ^{2}+4\alpha +2)}{2(\alpha +1)(2\alpha +1)}\end{array}
	\right] ,
\end{aligned}
\end{equation}
where $K_1$ was defined in (\ref{K1}). Note that these matrices have dimension $(p+1)\times (p+1)$ where $p$ is the dimension of vector $\boldsymbol{\beta }$ for each model. In our example, $p=4$ and therefore,
$$ \boldsymbol{\Omega }_{n}\left( \widehat{\boldsymbol{\theta }}_{\alpha }^s\right) \boldsymbol{\Psi }_{s, n}^{-1}\left( \widehat{\boldsymbol{\beta }}_{\alpha }^s,\widehat{\sigma }_{\alpha }^s\right) = (\widehat{\sigma }_{\alpha }^s)^{2}K_{1}\frac{\left( \alpha +1\right) ^{\frac{3}{2}}}{\left( 2\alpha +1\right) ^{\frac{3}{2}}} \left[ \begin{array}{cc} \boldsymbol{I}_{p\times p} & \boldsymbol{0} \\ \boldsymbol{0}^T & \frac{3\alpha ^{2}+4\alpha +2}{(\alpha +1)^2(2\alpha +1)} \end{array} \right] ,$$
and hence,
\begin{equation*}
trace\left( \boldsymbol{\Omega }_{n}\left( \widehat{\boldsymbol{\theta }}_{\alpha }^s\right) \boldsymbol{\Psi }_{s, n}^{-1}\left( \widehat{\boldsymbol{\beta }}_{\alpha }^s,\widehat{\sigma }_{\alpha }^s\right) \right) = (\widehat{\sigma }_{\alpha }^s)^{2}K_{1}\left( p\frac{\left( \alpha +1\right) ^{\frac{3}{2}}}{\left( 2\alpha +1\right) ^{\frac{3}{2}}}+\frac{\left( \alpha +1\right) ^{\frac{1}{2}}\left( 3\alpha ^{2}+4\alpha +2\right) }{2\left( 2\alpha +1\right) ^{5/2}}\right) .
\end{equation*}

Therefore, applying the $RP_{NH}-$Criterion defined in (\ref{def-Renyi-criterion})
\begin{align}
RP_{NH}(M^{(s)}_{1},..., M^{(s)}_{n}, & \widehat{\boldsymbol{\beta }}_{\alpha }^s, \widehat{\sigma }_{\alpha }^2) \nonumber \\
 = & - {1\over \alpha }\left( \frac{1+\alpha }{2\pi }\right) ^{\frac{\alpha }{2(\alpha +1)}}\frac{1}{n}\sum_{i=1}^{n}(\widehat{\sigma }_{\alpha }^s)^{-\frac{\alpha }{\alpha +1}}\exp \left( -\frac{\alpha }{2}\left( \frac{Y_{i}-\boldsymbol{X}_{s, i}^{T}\widehat{\boldsymbol{\beta }}_{\alpha }^s}{\widehat{\sigma }_{\alpha }^s}\right) ^{2}\right) \nonumber \\
 & + {1\over \alpha }
+ \frac{1}{n}(\widehat{\sigma }_{\alpha }^s)^{2}K_{1}\left( p\frac{\left( \alpha +1\right) ^{\frac{3}{2}}}{\left( 2\alpha +1\right) ^{\frac{3}{2}}}+ \frac{\left( \alpha +1\right) ^{\frac{1}{2}}\left( 3\alpha ^{2}+4\alpha
+2\right) }{2\left( 2\alpha +1\right) ^{5/2}}\right) .\label{ali}
\end{align}

Finally, we select the model with minimum, in $s$, $RP_{NH}(M^{(s)}_{1},..., M^{(s)}_{n}, \widehat{\boldsymbol{\beta }}_{\alpha }^s, \widehat{\sigma }_{\alpha }^s)$ as the most appropriate model among the four candidates.

\section{The restricted model}

Let us consider a particular case of the model selection problem. In some situations it is interesting to compare a full model based on $\boldsymbol{\theta }\in \boldsymbol{\Theta }\subset \mathbb{R}^{p}$, with $p$
parameters with other restricted models where the parameter has to satisfy additionally linear constraints of the form
\begin{equation}
\left\{ \boldsymbol{\theta }\in \boldsymbol{\Theta }/\text{ }\boldsymbol{m}(\boldsymbol{\theta })=\boldsymbol{0}_{r}\right\} ,  \label{3}
\end{equation}
where $\boldsymbol{0}_{r}$ denotes the null vector of dimension $r$ with $r<p$ and $\boldsymbol{m}:\mathbb{R}^{p}\rightarrow \mathbb{R}^{r}$ is a vector-valued function such that the $p\times r$ matrix
\begin{equation}
\mathbf{M}\left( \boldsymbol{\theta }\right) =\frac{\partial \boldsymbol{m}^{T}(\boldsymbol{\theta )}}{\partial \boldsymbol{\theta }}  \label{4}
\end{equation}
exists and is continuous in $\boldsymbol{\theta },$ and rank$\left( \mathbf{M}\left( \boldsymbol{\theta }\right) \right) =r, \forall \boldsymbol{\theta }\in \boldsymbol{\Theta }.$ Related to the divergence-based restricted estimation, in \cite{bamamapa18} the restricted minimum density power divergence estimator was defined. Later, in \cite{cajamapa23} the restricted MRPE for general populations was given.

Given a candidate model, we have already established that the best fitting parameter for this model based on the RP is defined by
\begin{equation*}
\boldsymbol{\theta }_{\boldsymbol{g}, \alpha }=\arg \min_{\boldsymbol{\theta \in \Theta \subset }\mathbb{R}^{p}}{} H_{\alpha }(\boldsymbol{\theta }),
\end{equation*}
where $H_{\alpha }(\boldsymbol{\theta })$ was defined in Eq. (\ref{A}). On the other hand, applying the same criterion for the restricted model, we obtain that the best-fitting parameter for the restricted model is given by

\begin{equation*}
\boldsymbol{\theta }_{\boldsymbol{g}, \alpha }^R=\arg \min_{\boldsymbol{\theta \in }\Theta /\text{ }\boldsymbol{m}(\boldsymbol{\theta })=\boldsymbol{0}_{r}}{} H_{\alpha }(\boldsymbol{\theta }).
\end{equation*}

Following similar arguments than in Section 2, we defined the restricted MRPE as follows.

\begin{definition}
	Given $Y_{1},...,Y_{n}$ be i.n.i.d.o., the {\bf restricted MRPE} (RMRPE), $\widetilde{\boldsymbol{\theta }}_{\alpha },$ is given by
	\begin{equation}
		\widetilde{\boldsymbol{\theta }}_{\alpha }=\arg \min_{\boldsymbol{\theta} \in \Theta /\boldsymbol{m}(\boldsymbol{\theta })=\boldsymbol{0}_{r}} H_{n, \alpha }(\boldsymbol{\theta }),  \label{3.4.1}
	\end{equation}
	with $H_{n, \alpha }(\boldsymbol{\theta })$ defined in (\ref{2.4}) for $\alpha >0$ and in (\ref{2.5}) for $\alpha = 0.$
\end{definition}

Note that

\begin{equation*}
H_{n, \alpha }(\widehat{\boldsymbol{\theta }}_{\alpha })\leq H_{n, \alpha }(\widetilde{\boldsymbol{\theta }}_{\alpha }).
\end{equation*}

The following theorem presents a representation of the RMPRE.

\begin{theorem}
Assume conditions {\bf C1}-{\bf C8} and suppose that $\boldsymbol{\theta }_{\boldsymbol{g}, \alpha }$ satisfies the conditions of the restricted model. Then,

\begin{equation*}
n^{1/2}(\widetilde{\boldsymbol{\theta }}_{\alpha }-\boldsymbol{\theta }_{\boldsymbol{g}, \alpha })=\boldsymbol{P}^{\ast }(\boldsymbol{\theta }_{\boldsymbol{g}, \alpha })n^{1/2}\left( \frac{\partial H_{n, \alpha }(\boldsymbol{\theta })}{\partial \boldsymbol{\theta }}\right) _{\boldsymbol{\theta =\theta }_{\boldsymbol{g}, \alpha }}+o_{p}(1),
\end{equation*}
being

\begin{equation}
\boldsymbol{P}^{\ast }(\boldsymbol{\theta }_{\boldsymbol{g}, \alpha }) = \boldsymbol{Q}_{\alpha }(\boldsymbol{\theta }_{\boldsymbol{g}, \alpha })\boldsymbol{M}(\boldsymbol{\theta }_{\boldsymbol{g}, \alpha })^{T}\boldsymbol{\Psi }_{n}\left( \boldsymbol{\theta }_{\boldsymbol{g}, \alpha }\right) ^{-1} - \boldsymbol{\Psi }_{n}\left( \boldsymbol{\theta }_{\boldsymbol{g}, \alpha }\right) ^{-1},  \label{P}
\end{equation}
with

\begin{equation}
\boldsymbol{Q}_{\alpha }(\boldsymbol{\theta }_{\boldsymbol{g}, \alpha })=\boldsymbol{\Psi }_{n}\left( \boldsymbol{\theta }_{\boldsymbol{g}, \alpha }\right) ^{-1}\boldsymbol{M}(\boldsymbol{\theta }_{\boldsymbol{g}, \alpha })\left[ \boldsymbol{M}(\boldsymbol{\theta }_{\boldsymbol{g}, \alpha })^{T}\boldsymbol{\Psi }_{n}\left( \boldsymbol{\theta }_{\boldsymbol{g}, \alpha }\right) ^{-1}\boldsymbol{M}(\boldsymbol{\theta }_{\boldsymbol{g}, \alpha })\right] ^{-1}. \label{Q}
\end{equation}
\end{theorem}

\begin{proof}
The RMRPE estimator of $\boldsymbol{\theta }$, $\widetilde{\boldsymbol{\theta }}_{\alpha }$, must satisfy
\begin{equation}
\left\{
\begin{array}{r}
\left( \tfrac{\partial H_{n, \alpha }(\boldsymbol{\theta })}{\partial \boldsymbol{\theta }}\right) _{\boldsymbol{\theta =}\widetilde{\boldsymbol{\theta }}_{\alpha }}+\boldsymbol{M}(\widetilde{\boldsymbol{\theta }}_{\alpha })\boldsymbol{\lambda }_{n}=\boldsymbol{0}_{p}, \\
\boldsymbol{m}(\widetilde{\boldsymbol{\theta }}_{\alpha })=\boldsymbol{0}_{r},
\end{array}
\right\} \Leftrightarrow \left\{
\begin{array}{r}
\left( \tfrac{\partial H_{n, \alpha }(\boldsymbol{\theta })}{\partial \boldsymbol{\theta }}\right) _{\boldsymbol{\theta =}\widetilde{\boldsymbol{\theta }}_{\alpha }} = -\boldsymbol{M}(\widetilde{\boldsymbol{\theta }}_{\alpha })\boldsymbol{\lambda }_{n} \\
\boldsymbol{m}(\widetilde{\boldsymbol{\theta }}_{\alpha })=\boldsymbol{0}_{r}
\end{array}
\right. , \label{EQ:lagrange_restrictions}
\end{equation}
where $\boldsymbol{\lambda }_{n}$ is a vector of Lagrangian multipliers. Now, applying Eq. (\ref{dist-asymp}), we can write $\widetilde{\boldsymbol{\theta }}_{\alpha }=\boldsymbol{\theta }_{\boldsymbol{g}, \alpha }+\boldsymbol{t}n^{-1/2}$, where $||\boldsymbol{t}||<c$, for some $0<c<\infty $. We have, applying Taylor,

\begin{equation*}  \label{MLE}
	\begin{aligned}
		\left( \frac{\partial H_{n, \alpha }(\boldsymbol{\theta })}{\partial \boldsymbol{\theta }}\right) _{\boldsymbol{\theta }=\widetilde{\boldsymbol{\theta }}_{\alpha }} = &
		\left( \frac{\partial H_{n, \alpha }(\boldsymbol{\theta })}{\partial \boldsymbol{\theta }}\right) _{\boldsymbol{\theta =\theta }_{\boldsymbol{g}, \alpha }}+\left( \frac{\partial^{2}H_{n, \alpha }(\boldsymbol{\theta })}{\partial \boldsymbol{\theta }\partial \boldsymbol{\theta }^{T}}\right) _{\boldsymbol{\theta =\theta }_{\boldsymbol{g}, \alpha }}(\widetilde{\boldsymbol{\theta }}_{\alpha }-\boldsymbol{\theta }_{\boldsymbol{g}, \alpha })\\
		&+o(||\widetilde{\boldsymbol{\theta }}_{\alpha }-\boldsymbol{\theta }_{\boldsymbol{g}, \alpha }||^{2}) ,
	\end{aligned}
\end{equation*}
and hence

\begin{equation*} \label{[A]}
\begin{aligned}
n^{1/2}\left( \frac{\partial H_{n, \alpha }(\boldsymbol{\theta })}{\partial \boldsymbol{\theta }}\right) _{\boldsymbol{\theta }= \widetilde{\boldsymbol{\theta }}_{\alpha }} =&
n^{1/2}\left( \frac{\partial H_{n, \alpha }(\boldsymbol{\theta })}{\partial \boldsymbol{\theta }}\right)_{\boldsymbol{\theta =\theta }_{\boldsymbol{g}, \alpha }}\\
&+\left( \frac{\partial ^{2}H_{n, \alpha }(\boldsymbol{\theta })}{\partial \boldsymbol{\theta }\partial \boldsymbol{\theta }^{T}}\right) _{\boldsymbol{\theta =\theta }_{\boldsymbol{g}, \alpha }}n^{1/2}(\widetilde{\boldsymbol{\theta }}_{\alpha }-\boldsymbol{\theta }_{\boldsymbol{g}, \alpha })+o(n^{1/2}||\widetilde{\boldsymbol{\theta }}_{\alpha }-\boldsymbol{\theta }_{\boldsymbol{g}, \alpha }||^{2}).
\end{aligned}
\end{equation*}

However,
\begin{equation*}
o(n^{1/2}||\widetilde{\boldsymbol{\theta }}_{\alpha }-\boldsymbol{\theta }_{\boldsymbol{g}, \alpha }||^{2})=o(n^{1/2}||\boldsymbol{t}||^{2}/n)=o(n^{-1/2}||\boldsymbol{t}||^{2})=o(O_{p}(1))=o_{p}(1).
\end{equation*}

Now,
\begin{equation*}
\begin{aligned}
	& \left( \frac{\partial ^{2}H_{n, \alpha }(\boldsymbol{\theta })}{\partial \theta _{j}\partial \theta _{k}}\right) _{\boldsymbol{\theta =\theta }_{\boldsymbol{g}, \alpha }} =
	 \frac{1}{n}\tsum\limits_{i=1}^{n}\left( \frac{\partial^{2}\hat{V}_{i}(Y_{i};\boldsymbol{\theta })}{\partial \theta _{j}\partial \theta_{k}}\right) _{\boldsymbol{\theta =\theta }_{\boldsymbol{g}, \alpha }}\\
	& \overset{P}{\underset{}{\longrightarrow }} \frac{1}{n}\tsum\limits_{i=1}^{n}E_{Y_i} \left[ \left( \frac{\partial ^{2}\hat{V}_{i}(Y;\boldsymbol{\theta })}{\partial \theta _{j}\partial \theta _{k}}\right) _{\boldsymbol{\theta
			=\theta }_{\boldsymbol{g}, \alpha }}\right] =\left( \boldsymbol{\Psi }_{n}\left( \boldsymbol{\theta }_{\boldsymbol{g}, \alpha }\right) \right) _{jk}.
\end{aligned}
\end{equation*}

Therefore,

\begin{equation}
n^{1/2}\left( \frac{\partial H_{n, \alpha }(\boldsymbol{\theta })}{\partial \boldsymbol{\theta }}\right) _{\boldsymbol{\theta }=\widetilde{\boldsymbol{\theta }}_{\alpha }} =
n^{1/2}\left( \frac{\partial H_{n, \alpha }(\boldsymbol{\theta })}{\partial \boldsymbol{\theta }}\right) _{\boldsymbol{\theta =\theta }_{\boldsymbol{g}, \alpha }}+\boldsymbol{\Psi }_{n}\left( \boldsymbol{\theta }_{\boldsymbol{g}, \alpha }\right) n^{1/2}(\widetilde{\boldsymbol{\theta }}_{\alpha }-\boldsymbol{\theta }_{\boldsymbol{g}, \alpha })+o_{p}(1).  \label{EQ:Theorem2_B}
\end{equation}

As the RMRPE $\widetilde{\boldsymbol{\theta }}_{\alpha }$ must satisfy the conditions in (\ref{EQ:lagrange_restrictions}), and in view of (\ref{EQ:Theorem2_B}) we have

\begin{equation*}
n^{1/2}\left( \frac{\partial H_{n, \alpha }(\boldsymbol{\theta })}{\partial \boldsymbol{\theta }}\right) _{\boldsymbol{\theta =\theta }_{\boldsymbol{g}, \alpha }} =
-\boldsymbol{\Psi }_{n}\left( \boldsymbol{\theta }_{\boldsymbol{g}, \alpha }\right) n^{1/2}(\widetilde{\boldsymbol{\theta }}_{\alpha }-\boldsymbol{\theta }_{\boldsymbol{g}, \alpha }) - \boldsymbol{M}(\widetilde{\boldsymbol{\theta }}_{\alpha })n^{1/2}\boldsymbol{\lambda }_{n}+o_{p}(1).
\end{equation*}

And applying the continuity of $\boldsymbol{M}$, this can be written as

\begin{equation}
-\boldsymbol{\Psi }_{n}\left( \boldsymbol{\theta }_{\boldsymbol{g}, \alpha }\right) n^{1/2}(\widetilde{\boldsymbol{\theta }}_{\alpha }-\boldsymbol{\theta }_{\boldsymbol{g}, \alpha }) - \boldsymbol{M}(\boldsymbol{\theta }_{\boldsymbol{g}, \alpha })n^{1/2}\boldsymbol{\lambda }_{n} =
n^{1/2}\left( \frac{\partial H_{n, \alpha }(\boldsymbol{\theta })}{\partial \boldsymbol{\theta }}\right) _{\boldsymbol{\theta =\theta }_{\boldsymbol{g}, \alpha }} + o_p(1).  \label{EQ:Theorem2_D}
\end{equation}

On the other hand, applying Taylor to $\boldsymbol{m},$ we obtain

\begin{equation}
n^{1/2}\boldsymbol{m}(\widetilde{\boldsymbol{\theta }}_{\alpha }) = n^{1/2}\boldsymbol{m}(\boldsymbol{\theta }_{\boldsymbol{g}, \alpha }) + \boldsymbol{M}(\boldsymbol{\theta }_{\boldsymbol{g}, \alpha })^{T}n^{1/2}(\widetilde{\boldsymbol{\theta }}_{\alpha }-\boldsymbol{\theta }_{\boldsymbol{g}, \alpha })+o_{p}(1).  \label{EQ:Theorem2_C}
\end{equation}

From (\ref{EQ:Theorem2_C}) and applying that $\boldsymbol{m}(\widetilde{\boldsymbol{\theta }}_{\alpha })=\boldsymbol{0}_r, \boldsymbol{m}(\boldsymbol{\theta }_{\boldsymbol{g}, \alpha })= \boldsymbol{0}_r,$ it follows that

\begin{equation}
\boldsymbol{M}(\boldsymbol{\theta }_{\boldsymbol{g}, \alpha })^{T}n^{1/2}(\widetilde{\boldsymbol{\theta }}_{\alpha }-\boldsymbol{\theta }_{\boldsymbol{g}, \alpha })+o_{p}(1)=\boldsymbol{0}_{r}.  \label{EQ:Theorem2_E}
\end{equation}

Now, we can express equations (\ref{EQ:Theorem2_D}) and (\ref{EQ:Theorem2_E}) in matrix form as

\begin{equation*}
\left(
\begin{array}{cc}
-\boldsymbol{\Psi }_{n}\left( \boldsymbol{\theta }_{\boldsymbol{g}, \alpha }\right) & -\boldsymbol{M}(\boldsymbol{\theta }_{\boldsymbol{g}, \alpha }) \\
\boldsymbol{M}(\boldsymbol{\theta }_{\boldsymbol{g}, \alpha })^{T} & \boldsymbol{0}_{r\times r}
\end{array}
\right) \left(
\begin{array}{c} n^{1/2}(\widetilde{\boldsymbol{\theta }}_{\alpha }-\boldsymbol{\theta }_{\boldsymbol{g}, \alpha }) \\ n^{1/2}\boldsymbol{\lambda }_{n} \end{array}
\right) =\left(
\begin{array}{c} n^{1/2}\left( \frac{\partial H_{n}(\boldsymbol{\theta })}{\partial \boldsymbol{\theta }}\right) _{\boldsymbol{\theta =\theta }_{\boldsymbol{g}, \alpha }} \\ \boldsymbol{0}_{r}\end{array}
\right) +o_{p}(1).
\end{equation*}

Therefore,

\begin{equation*}
\left(
\begin{array}{c} n^{1/2}(\widetilde{\boldsymbol{\theta }}_{\alpha }-\boldsymbol{\theta }_{\boldsymbol{g}, \alpha }) \\ n^{1/2}\boldsymbol{\lambda }_{n} \end{array}
\right) =\left(
\begin{array}{cc}
-\boldsymbol{\Psi }_{n}\left( \boldsymbol{\theta }_{\boldsymbol{g}, \alpha }\right) & -\boldsymbol{M}(\boldsymbol{\theta }_{\boldsymbol{g}, \alpha }) \\
\boldsymbol{M}(\boldsymbol{\theta }_{\boldsymbol{g}, \alpha })^{T} & \boldsymbol{0}_{r\times r}
\end{array}
\right) ^{-1}\left(
\begin{array}{c} n^{1/2}\left( \frac{\partial H_{n}(\boldsymbol{\theta })}{\partial \boldsymbol{\theta }}\right) _{\boldsymbol{\theta =\theta }_{\boldsymbol{g}, \alpha }} \\ \boldsymbol{0}_{r} \end{array}
\right) +o_{p}(1).
\end{equation*}

But

\begin{equation*}
\left(
\begin{array}{cc}
-\boldsymbol{\Psi }_{n}\left( \boldsymbol{\theta }_{\boldsymbol{g}, \alpha }\right)  & -\boldsymbol{M}(\boldsymbol{\theta }_{\boldsymbol{g}, \alpha }) \\
\boldsymbol{M}(\boldsymbol{\theta }_{\boldsymbol{g}, \alpha })^{T} & \boldsymbol{0}
\end{array}
\right) ^{-1}={\left(
\begin{array}{cc}
\boldsymbol{P}_{\alpha }^{\ast }(\boldsymbol{\theta }_{\boldsymbol{g, \alpha }}) & -\boldsymbol{Q}_{\alpha }(\boldsymbol{\theta }_{\boldsymbol{g}, \alpha }) \\
- \boldsymbol{Q}_{\alpha }(\boldsymbol{\theta }_{\boldsymbol{g}, \alpha })^{T} & \boldsymbol{R}_{\alpha }(\boldsymbol{\theta }_{\boldsymbol{g}, \alpha })
\end{array}
\right) },
\end{equation*}
where $\boldsymbol{P}_{\alpha }^{\ast }(\boldsymbol{\theta }_{\boldsymbol{g}, \alpha })$ and $\boldsymbol{Q}_{\alpha }(\boldsymbol{\theta }_{\boldsymbol{g}, \alpha })$ are given in (\ref{P}) and (\ref{Q}), respectively. The matrix $\boldsymbol{R}_{\alpha }(\boldsymbol{\theta }_{\boldsymbol{g}, \alpha })$ is the matrix needed to make the right hand side of the above equation equal to the indicated inverse. Then,

\begin{equation}
n^{1/2}(\widetilde{\boldsymbol{\theta }}_{\alpha }-\boldsymbol{\theta }_{\boldsymbol{g}, \alpha }) =
\boldsymbol{P}^{\ast }(\boldsymbol{\theta }_{\boldsymbol{g}, \alpha })n^{1/2}\left( \frac{\partial H_{n}(\boldsymbol{\theta })}{\partial \boldsymbol{\theta }}\right) _{\boldsymbol{\theta =\theta }_{\boldsymbol{g}, \alpha }}+o_{p}(1)  \label{E}
\end{equation}
and the result holds.
\end{proof}

In the following lemma we establish a property about matrix $\boldsymbol{P}_{\alpha }^{\ast }(\boldsymbol{\theta }_{\boldsymbol{g}, \alpha })$ that will be required for the next theorem.

\begin{lemma}
Given $\boldsymbol{P}_{\alpha }^{\ast }(\boldsymbol{\theta }_{\boldsymbol{g}, \alpha })$ and $\boldsymbol{\Psi }_{n}\left( \boldsymbol{\theta }_{\boldsymbol{g}, \alpha }\right)$, it follows

$$ \boldsymbol{P}_{\alpha }^{\ast }(\boldsymbol{\theta }_{\boldsymbol{g}, \alpha }) \boldsymbol{\Psi }_{n}\left( \boldsymbol{\theta }_{\boldsymbol{g}, \alpha }\right) \boldsymbol{P}_{\alpha }^{\ast }(\boldsymbol{\theta }_{\boldsymbol{g}, \alpha }) = -\boldsymbol{P}_{\alpha }^{\ast }(\boldsymbol{\theta }_{\boldsymbol{g}, \alpha }).$$
\end{lemma}

\begin{proof}
Applying the definitions and denoting

$$ \boldsymbol{A}^{-1} (\boldsymbol{\theta }_{\boldsymbol{g}, \alpha }) = \left[ \boldsymbol{M}(\boldsymbol{\theta }_{\boldsymbol{g}, \alpha })^{T} \boldsymbol{\Psi }_{n}\left( \boldsymbol{\theta }_{\boldsymbol{g}, \alpha }\right) ^{-1} \boldsymbol{M}(\boldsymbol{\theta }_{\boldsymbol{g}, \alpha })\right]^{-1},$$
we obtain

\begin{align*}
\boldsymbol{P}_{\alpha }^{\ast }&(\boldsymbol{\theta }_{\boldsymbol{g}, \alpha }) \boldsymbol{\Psi }_{n}\left( \boldsymbol{\theta }_{\boldsymbol{g}, \alpha }\right) \boldsymbol{P}_{\alpha }^{\ast }(\boldsymbol{\theta }_{\boldsymbol{g}, \alpha })\\
 = & \left[ \boldsymbol{\Psi }_{n}\left( \boldsymbol{\theta }_{\boldsymbol{g}, \alpha }\right) ^{-1} \boldsymbol{M}(\boldsymbol{\theta }_{\boldsymbol{g}, \alpha }) \boldsymbol{A}^{-1}(\boldsymbol{\theta }_{\boldsymbol{g}, \alpha }) \boldsymbol{M}(\boldsymbol{\theta }_{\boldsymbol{g}, \alpha })^{T}\boldsymbol{\Psi }_{n}\left(
\boldsymbol{\theta }_{\boldsymbol{g}, \alpha }\right) ^{-1} - \boldsymbol{\Psi }_{n}\left( \boldsymbol{\theta }_{\boldsymbol{g}, \alpha }\right) ^{-1} \right] \\
 & \boldsymbol{\Psi }_{n}\left( \boldsymbol{\theta }_{\boldsymbol{g}, \alpha }\right) \boldsymbol{P}_{\alpha }^{\ast }(\boldsymbol{\theta }_{\boldsymbol{g}, \alpha }) \\
 = & \left[ \boldsymbol{\Psi }_{n}\left( \boldsymbol{\theta }_{\boldsymbol{g}, \alpha }\right) ^{-1} \boldsymbol{M}(\boldsymbol{\theta }_{\boldsymbol{g}, \alpha }) \boldsymbol{A}^{-1}(\boldsymbol{\theta }_{\boldsymbol{g}, \alpha }) \boldsymbol{M}(\boldsymbol{\theta }_{\boldsymbol{g}, \alpha })^T - Id \right] \boldsymbol{P}_{\alpha }^{\ast }(\boldsymbol{\theta }_{\boldsymbol{g}, \alpha }) \\
 =
 & \boldsymbol{\Psi }_{n}\left( \boldsymbol{\theta }_{\boldsymbol{g}, \alpha }\right) ^{-1} \boldsymbol{M}(\boldsymbol{\theta }_{\boldsymbol{g}, \alpha }) \boldsymbol{A}^{-1}\left( \boldsymbol{\theta }_{\boldsymbol{g}, \alpha }\right) \boldsymbol{M}(\boldsymbol{\theta }_{\boldsymbol{g}, \alpha })^T \boldsymbol{\Psi }_{n}\left( \boldsymbol{\theta }_{\boldsymbol{g}, \alpha }\right) ^{-1} \boldsymbol{M}(\boldsymbol{\theta }_{\boldsymbol{g}, \alpha }) \boldsymbol{A}^{-1}(\boldsymbol{\theta }_{\boldsymbol{g}, \alpha }) \boldsymbol{\Psi }_{n}\left( \boldsymbol{\theta }_{\boldsymbol{g}, \alpha }\right) ^{-1} \\
& - \boldsymbol{\Psi }_{n}\left( \boldsymbol{\theta }_{\boldsymbol{g}, \alpha }\right) ^{-1} \boldsymbol{M}(\boldsymbol{\theta }_{\boldsymbol{g}, \alpha }) \boldsymbol{A}^{-1}(\boldsymbol{\theta }_{\boldsymbol{g}, \alpha }) \boldsymbol{M}(\boldsymbol{\theta }_{\boldsymbol{g}, \alpha })^T \boldsymbol{\Psi }_{n}\left( \boldsymbol{\theta }_{\boldsymbol{g}, \alpha }\right) ^{-1} - \boldsymbol{P}_{\alpha }^{\ast }(\boldsymbol{\theta }_{\boldsymbol{g}, \alpha })
 \\
 = & - \boldsymbol{P}_{\alpha }^{\ast }(\boldsymbol{\theta }_{\boldsymbol{g}, \alpha }).
\end{align*}

Hence, the result holds.
\end{proof}

Suppose now that we have chosen a model as the best fitting model and we wonder if this model overfits the data and a restricted model is more accurate. Then, we can pose this problem as a model selection problem with two models, the big one and a restricted model, and apply the results of the previous section. Hence, it suffices to compute $RP_{NH}((M^{(s)}_{1},..., M^{(s)}_{n}, \boldsymbol{\theta })$ for both models and select the one attaining the minimum. Assuming the restricted model is correct, in the following theorem we shall establish the asymptotic distribution of
\begin{equation*}
2 n\left[ RP_{NH}\left( M^{(s)}_{1},..., M^{(s)}_{n}, \widehat{\boldsymbol{\theta }}_{\alpha } \right) - RP_{NH}\left( M^{(s)}_{1},..., M^{(s)}_{n}, \widetilde{\boldsymbol{\theta }}_{\alpha }\right) \right] ,
\end{equation*}
where $RP_{NH}\left( (M^{(s)}_{1},..., M^{(s)}_{n}, \widehat{\boldsymbol{\theta }}_{\alpha }\right) $ was given in (\ref{CRP}) and
\begin{equation*}
RP_{NH}\left( (M^{(s)}_{1},..., M^{(s)}_{n}, \widetilde{\boldsymbol{\theta }}_{\alpha }\right) =  H_{n,\alpha }\left( \widetilde{\boldsymbol{\theta }}_{\alpha }\right) + \frac{1}{n}trace\left( \boldsymbol{\Omega }
_{n}^{R}\left( \widetilde{\boldsymbol{\theta }}_{\alpha }\right) \boldsymbol{\Psi }_{n}^{R}\left( \widetilde{\boldsymbol{\theta }}_{\alpha }\right) ^{-1}\right) ,
\end{equation*}%
being $\boldsymbol{\Psi }_{n}^{R}\left( \widetilde{\boldsymbol{\theta }}_{\alpha }\right) $ and $\boldsymbol{\Omega }_{n}^{R}\left( \widetilde{\boldsymbol{\theta }}_{\alpha }\right) $ the matrices defined in (\ref{2.4.2}) and (\ref{2.4.3}) but for the restricted model.

Note that the probability of selecting the restricted model is
\begin{equation*}
\Pr \left( RP_{NH}\left( M_{1}^{(k)},...,M_{n}^{(k)},\widehat{\boldsymbol{\theta }}_{\alpha }\right) - RP_{NH}\left( M_{1}^{(k)},...,M_{n}^{(k)}, \widetilde{\boldsymbol{\theta }}_{\alpha } \right)  >0\right) .
\end{equation*}

\begin{theorem}
Assume conditions {\bf C1-C8} hold and suppose that the fitting parameter, $\boldsymbol{\theta }_{\boldsymbol{g}, \alpha },$ belongs to the restricted model. Then, the asymptotic distribution of

\begin{equation*}
2 n\left( RP_{NH}\left( (M^{(s)}_{1},..., M^{(s)}_{n}, \widehat{\boldsymbol{\theta }}_{\alpha } \right) -RP_{NH}\left( (M^{(s)}_{1},..., M^{(s)}_{n}, \widetilde{\boldsymbol{\theta }}_{\alpha }\right) \right)
\end{equation*}
coincides with the distribution of the random variable

\begin{equation*}
\sum_{j=1}^{r}\lambda _{j}(\boldsymbol{\theta }_{g,\alpha })\left( \boldsymbol{\theta }\right) Z_{j}^{2} + 2 trace(\boldsymbol{\Omega }_{n}\left( \boldsymbol{\theta }_{\boldsymbol{g}, \alpha }\right) \boldsymbol{\Psi }_{n}^{-1}\left( \boldsymbol{\theta }_{\boldsymbol{g}, \alpha }\right) ) - 2 trace( \boldsymbol{\Omega }_{n}^{R}\left( \boldsymbol{\theta }_{\boldsymbol{g}, \alpha }\right)(\boldsymbol{\Psi }_{n}^R)^{-1}\left( \boldsymbol{\theta }_{\boldsymbol{g}, \alpha }\right)  ) ,
\end{equation*}
where $Z_{1},\ldots ,Z_{k}$ are independent standard normal variables, $\lambda _{1}(\boldsymbol{\theta }_{\boldsymbol{g}, \alpha }),\ldots ,\lambda _{r}(\boldsymbol{\theta }_{\boldsymbol{g}, \alpha })$ are the nonzero eigenvalues of $- \boldsymbol{Q}_{\alpha }(\boldsymbol{\theta }_{\boldsymbol{g}, \alpha })\boldsymbol{M}(\boldsymbol{\theta }_{\boldsymbol{g}, \alpha })^{T}\boldsymbol{\Psi }_{n}\left( \boldsymbol{\theta }_{\boldsymbol{g}, \alpha }\right) ^{-1}\boldsymbol{\Omega }_{n}\left( \boldsymbol{\theta }_{\boldsymbol{g}, \alpha }\right) $ and

\begin{equation*}
r=rank\left( \boldsymbol{\Omega }_{n}\left( \boldsymbol{\theta }_{\boldsymbol{g}, \alpha }\right) \boldsymbol{Q}_{\alpha }(\boldsymbol{\theta }_{\boldsymbol{g}, \alpha })\boldsymbol{M}(\boldsymbol{\theta }_{\boldsymbol{g}, \alpha })^{T}\boldsymbol{\Psi }_{n}\left( \boldsymbol{\theta }_{\boldsymbol{g}, \alpha }\right)^{-1}\boldsymbol{\Omega }_{n}\left( \boldsymbol{\theta }_{\boldsymbol{g}, \alpha }\right) \right) .
\end{equation*}
\end{theorem}

\begin{proof}
Let us denote

\begin{equation*}
L=2n\left[ RP_{NH}\left( M^{(s)}_{1},..., M^{(s)}_{n}, \widehat{\boldsymbol{\theta }}_{\alpha } \right) -RP_{NH}\left( M^{(s)}_{1},..., M^{(s)}_{n}, \widetilde{\boldsymbol{\theta }}_{\alpha }\right) \right] .
\end{equation*}

Then,

\begin{equation*}
L=2n\left[ H_{n,\alpha }\left( \widehat{\boldsymbol{\theta }}_{\alpha }\right)  - H_{n,\alpha }\left( \widetilde{\boldsymbol{\theta }}_{\alpha }\right) \right] + 2 trace \left[ \boldsymbol{\Omega }_{n}\left( \widehat{\boldsymbol{\theta }}_{\alpha }\right) \boldsymbol{\Psi }_{n}\left( \widehat{\boldsymbol{\theta }}_{\alpha }\right)^{-1} \right] -2 trace\left[ \boldsymbol{\Omega }_{n}^R\left( \widetilde{\boldsymbol{\theta }}_{\alpha }\right) \boldsymbol{\Psi }_{n}^R\left( \widetilde{\boldsymbol{\theta }}_{\alpha }\right)^{-1}\right] .
\end{equation*}

First, note that

\begin{eqnarray*}
H_{n}\left( \widetilde{\boldsymbol{\theta }}_{\alpha }\right)
& = & H_{n, \alpha }\left( \boldsymbol{\theta }_{\boldsymbol{g}, \alpha }\right) + \left( {\partial H_{n, \alpha }\left( \boldsymbol{\theta }\right) \over \partial \boldsymbol{\theta }}\right)_{\boldsymbol{\theta } = \boldsymbol{\theta }_{\boldsymbol{g}, \alpha }} \left( \widetilde{\boldsymbol{\theta }}_{\alpha } - \boldsymbol{\theta }_{\boldsymbol{g}, \alpha }\right) \\
& & + {1\over 2} \left( \widetilde{\boldsymbol{\theta }}_{\alpha } - \boldsymbol{\theta }_{\boldsymbol{g}, \alpha }\right)^T \left( {\partial^2 H_{n, \alpha }\left( \boldsymbol{\theta }\right) \over \partial \boldsymbol{\theta } \partial \boldsymbol{\theta }^T}\right)_{\boldsymbol{\theta } = \boldsymbol{\theta }_{\boldsymbol{g}, \alpha }} \left( \widetilde{\boldsymbol{\theta }}_{\alpha } - \boldsymbol{\theta }_{\boldsymbol{g}, \alpha }\right) + o(||\widetilde{\boldsymbol{\theta }}_{\alpha } - \boldsymbol{\theta }_{\boldsymbol{g}, \alpha }||^2).
\end{eqnarray*}

Hence,

\begin{eqnarray*}
2n \left[ H_{n}\left( \widetilde{\boldsymbol{\theta }}_{\alpha }\right) - H_{n, \alpha }\left( \boldsymbol{\theta }_{\boldsymbol{g}, \alpha }\right) \right]
& = & 2 \sqrt{n} \left( {\partial H_{n, \alpha }\left( \boldsymbol{\theta }\right) \over \partial \boldsymbol{\theta }}\right)_{\boldsymbol{\theta } = \boldsymbol{\theta }_{\boldsymbol{g}, \alpha }} \sqrt{n} \left( \widetilde{\boldsymbol{\theta }}_{\alpha } - \boldsymbol{\theta }_{\boldsymbol{g}, \alpha }\right) \\
& & + \sqrt{n} \left( \widetilde{\boldsymbol{\theta }}_{\alpha } - \boldsymbol{\theta }_{\boldsymbol{g}, \alpha }\right)^T \left( {\partial^2 H_{n, \alpha }\left( \boldsymbol{\theta }\right) \over \partial \boldsymbol{\theta } \partial \boldsymbol{\theta }^T}\right)_{\boldsymbol{\theta } = \boldsymbol{\theta }_{\boldsymbol{g}, \alpha }} \sqrt{n} \left( \widetilde{\boldsymbol{\theta }}_{\alpha } - \boldsymbol{\theta }_{\boldsymbol{g}, \alpha }\right) + o_p(1).
\end{eqnarray*}

Now, taking into account that

$$ \sqrt{n} \left( \widetilde{\boldsymbol{\theta }}_{\alpha } - \boldsymbol{\theta }_{\boldsymbol{g}, \alpha }\right) = \boldsymbol{P}_{\alpha }^{\ast }(\boldsymbol{\theta }_{\boldsymbol{g}, \alpha }) \sqrt{n}\left( {\partial H_{n, \alpha }\left( \boldsymbol{\theta }\right) \over \partial \boldsymbol{\theta }}\right)_{\boldsymbol{\theta } = \boldsymbol{\theta }_{\boldsymbol{g}, \alpha }} + o_p(1),$$
and

$$ \left( {\partial^2 H_{n, \alpha }\left( \boldsymbol{\theta }\right) \over \partial \boldsymbol{\theta } \partial \boldsymbol{\theta }^T}\right)_{\boldsymbol{\theta } = \boldsymbol{\theta }_{\boldsymbol{g}, \alpha }} \rightarrow \boldsymbol{\Psi }_{n}\left( \boldsymbol{\theta }_{\boldsymbol{g}, \alpha }\right) ,$$
by Eq. (\ref{Ecuacion}), we conclude that

\begin{equation*}
	\begin{aligned}
2 n& \left[ H_{n}\left( \widetilde{\boldsymbol{\theta }}_{\alpha }\right) - H_{n, \alpha }\left( \boldsymbol{\theta }_{\boldsymbol{g}, \alpha }\right) \right]\\
= & 2 \sqrt{n} \left( {\partial H_{n, \alpha }\left( \boldsymbol{\theta }\right) \over \partial \boldsymbol{\theta }}\right)_{\boldsymbol{\theta } = \boldsymbol{\theta }_{\boldsymbol{g}, \alpha }}^T \boldsymbol{P}_{\alpha }^{\ast }(\boldsymbol{\theta }_{\boldsymbol{g}, \alpha }) \sqrt{n}\left( {\partial H_{n, \alpha }\left( \boldsymbol{\theta }\right) \over \partial \boldsymbol{\theta }}\right)_{\boldsymbol{\theta } = \boldsymbol{\theta }_{\boldsymbol{g}, \alpha }} \\
 & + \sqrt{n} \left( {\partial H_{n, \alpha }\left( \boldsymbol{\theta }\right) \over \partial \boldsymbol{\theta }}\right)_{\boldsymbol{\theta } = \boldsymbol{\theta }_{\boldsymbol{g}, \alpha }}^T \boldsymbol{P}_{\alpha }^{\ast }(\boldsymbol{\theta }_{\boldsymbol{g}, \alpha }) \boldsymbol{\Psi }_{n}\left( \boldsymbol{\theta }_{\boldsymbol{g}, \alpha }\right) \boldsymbol{P}_{\alpha }^{\ast }(\boldsymbol{\theta }_{\boldsymbol{g}, \alpha }) \sqrt{n}\left( {\partial H_{n, \alpha }\left( \boldsymbol{\theta }\right) \over \partial \boldsymbol{\theta }}\right)_{\boldsymbol{\theta } = \boldsymbol{\theta }_{\boldsymbol{g}, \alpha }} + o_p(1).
\end{aligned}
\end{equation*}

Now, applying the previous lemma, we know that
$$ \boldsymbol{P}_{\alpha }^{\ast }(\boldsymbol{\theta }_{\boldsymbol{g}, \alpha }) \boldsymbol{\Psi }_{n}\left( \boldsymbol{\theta }_{\boldsymbol{g}, \alpha }\right) \boldsymbol{P}_{\alpha }^{\ast }(\boldsymbol{\theta }_{\boldsymbol{g}, \alpha }) = -\boldsymbol{P}_{\alpha }^{\ast }(\boldsymbol{\theta }_{\boldsymbol{g}, \alpha }),$$
and thus,

$$ 2n \left[ H_{n}\left( \widetilde{\boldsymbol{\theta }}_{\alpha }\right) - H_{n, \alpha }\left( \boldsymbol{\theta }_{\boldsymbol{g}, \alpha }\right) \right] = \sqrt{n} \left( {\partial H_{n, \alpha }\left( \boldsymbol{\theta }\right) \over \partial \boldsymbol{\theta }}\right)_{\boldsymbol{\theta } = \boldsymbol{\theta }_{\boldsymbol{g}, \alpha }}^T \boldsymbol{P}_{\alpha }^{\ast }(\boldsymbol{\theta }_{\boldsymbol{g}, \alpha }) \sqrt{n}\left( {\partial H_{n, \alpha }\left( \boldsymbol{\theta }\right) \over \partial \boldsymbol{\theta }}\right)_{\boldsymbol{\theta } = \boldsymbol{\theta }_{\boldsymbol{g}, \alpha }} + o_p(1).$$

On the other hand,
\begin{eqnarray*}
H_{n, \alpha }\left( \boldsymbol{\theta }_{\boldsymbol{g}, \alpha }\right)
& = & H_{n}\left( \widehat{\boldsymbol{\theta }}_{\alpha }\right) + \left( {\partial H_{n, \alpha }\left( \boldsymbol{\theta }\right) \over \partial \boldsymbol{\theta }}\right)_{\boldsymbol{\theta } = \widehat{\boldsymbol{\theta }}_{\alpha }} \left( \boldsymbol{\theta }_{\boldsymbol{g}, \alpha } - \widehat{\boldsymbol{\theta }}_{\alpha } \right) \\
& & + {1\over 2} \left( \boldsymbol{\theta }_{\boldsymbol{g}, \alpha }- \widehat{\boldsymbol{\theta }}_{\alpha }\right)^T \left( {\partial^2 H_{n, \alpha }\left( \boldsymbol{\theta }\right) \over \partial \boldsymbol{\theta } \partial \boldsymbol{\theta }^T}\right)_{\boldsymbol{\theta } = \widehat{\boldsymbol{\theta }}_{\alpha }} \left( \boldsymbol{\theta }_{\boldsymbol{g}, \alpha }-\widehat{\boldsymbol{\theta }}_{\alpha } \right) + o(||\boldsymbol{\theta }_{\boldsymbol{g}, \alpha }-\widehat{\boldsymbol{\theta }}_{\alpha }||^2).
\end{eqnarray*}
Now, taking into account that
$$ \left( {\partial^2 H_{n, \alpha }\left( \boldsymbol{\theta }\right) \over \partial \boldsymbol{\theta } \partial \boldsymbol{\theta }^T}\right)_{\boldsymbol{\theta } = \widehat{\boldsymbol{\theta }}_{\alpha }}  \longrightarrow \left( {\partial^2 H_{n, \alpha }\left( \boldsymbol{\theta }\right) \over \partial \boldsymbol{\theta } \partial \boldsymbol{\theta }^T}\right)_{\boldsymbol{\theta } = \boldsymbol{\theta }_{\boldsymbol{g}, \alpha }} \longrightarrow \boldsymbol{\Psi }_{n}\left( \boldsymbol{\theta }_{\boldsymbol{g}, \alpha }\right) ,$$
and
$$ \left( {\partial H_{n, \alpha }\left( \boldsymbol{\theta }\right) \over \partial \boldsymbol{\theta }}\right)_{\boldsymbol{\theta } = \widehat{\boldsymbol{\theta }}_{\alpha }} =0,$$
we conclude that

$$ 2n \left[ H_{n}\left( \widehat{\boldsymbol{\theta }}_{\alpha }\right) - H_{n, \alpha }\left( \boldsymbol{\theta }_{\boldsymbol{g}, \alpha }\right) \right] = - \sqrt{n} \left( \widehat{\boldsymbol{\theta }}_{\alpha } -
\boldsymbol{\theta }_{\boldsymbol{g}, \alpha }\right)^T \boldsymbol{\Psi }_{n}\left( \boldsymbol{\theta }_{\boldsymbol{g}, \alpha }\right) \sqrt{n} \left( \widehat{\boldsymbol{\theta }}_{\alpha } -
\boldsymbol{\theta }_{\boldsymbol{g}, \alpha }\right) + o_p(1).$$

Applying $\left( \frac{\partial H_{n, \alpha }\left( \boldsymbol{\theta }\right) }{\partial \boldsymbol{\theta }}\right) _{\boldsymbol{\theta }=\widehat{\boldsymbol{\theta }}_{\alpha }} = \boldsymbol{0},$  we have by Taylor
$$ \boldsymbol{0} = n^{1/2} \left( \frac{\partial H_{n, \alpha }\left( \boldsymbol{\theta }\right) }{\partial \boldsymbol{\theta }^{T}}\right) _{\boldsymbol{\theta =\theta }_{\boldsymbol{g}, \alpha }} + \boldsymbol{\Psi }_{n}(\boldsymbol{\theta}_{\boldsymbol{g}, \alpha }) n^{1/2} \left( \widehat{\boldsymbol{\theta }}_{\alpha }-\boldsymbol{\theta }_{\boldsymbol{g}, \alpha }\right) +o_{p}(1), $$
so that
$$ n^{1/2} \left( \widehat{\boldsymbol{\theta }}_{\alpha }-\boldsymbol{\theta }_{\boldsymbol{g}, \alpha }\right) = - n^{1/2} \boldsymbol{\Psi }_{n}(\boldsymbol{\theta}_{\boldsymbol{g}, \alpha })^{-1} \left( \frac{\partial H_{n, \alpha }\left( \boldsymbol{\theta }\right) }{\partial \boldsymbol{\theta }^{T}}\right) _{\boldsymbol{\theta =\theta }_{\boldsymbol{g}, \alpha }} +o_{p}(1).$$
Hence,
$$ 2n \left[ H_{n}\left( \widehat{\boldsymbol{\theta }}_{\alpha }\right) - H_{n, \alpha }\left( \boldsymbol{\theta }_{\boldsymbol{g}, \alpha }\right) \right] = - \sqrt{n} \left( {\partial H_{n, \alpha }\left( \boldsymbol{\theta }\right) \over \partial \boldsymbol{\theta }}\right)_{\boldsymbol{\theta } = \boldsymbol{\theta }_{\boldsymbol{g}, \alpha }}^T \boldsymbol{\Psi }_{n}\left( \boldsymbol{\theta }_{\boldsymbol{g}, \alpha }\right)^{-1}  \sqrt{n}\left( {\partial H_{n, \alpha }\left( \boldsymbol{\theta }\right) \over \partial \boldsymbol{\theta }}\right)_{\boldsymbol{\theta } = \boldsymbol{\theta }_{\boldsymbol{g}, \alpha }} + o_p(1).$$

But as $\boldsymbol{P}^{\ast }(\boldsymbol{\theta }_{\boldsymbol{g}, \alpha })=\boldsymbol{Q}_{\alpha }(\boldsymbol{\theta }_{\boldsymbol{g}, \alpha })\boldsymbol{M}(\boldsymbol{\theta }_{\boldsymbol{g}, \alpha })^{T}\boldsymbol{\Psi }_{n}\left( \boldsymbol{\theta }_{\boldsymbol{g}, \alpha }\right) ^{-1} - \boldsymbol{\Psi }_{n}\left( \boldsymbol{\theta }_{\boldsymbol{g}, \alpha }\right)^{-1},$ we obtain

\begin{equation*}
	\begin{aligned}
2n\left[ H_{n, \alpha }\left( \widehat{\boldsymbol{\theta }}_{\alpha }\right) - H_{n, \alpha }\left( \widetilde{\boldsymbol{\theta }}_{\alpha }\right) \right]
 = &
- \sqrt{n}\left( \frac{\partial H_{n, \alpha }\left( \boldsymbol{\theta }\right) }{\partial \boldsymbol{\theta }^{T}}\right) _{\boldsymbol{\theta =\theta }_{\boldsymbol{g}, \alpha }}^T\boldsymbol{Q}_{\alpha }(\boldsymbol{\theta }_{\boldsymbol{g}, \alpha })\boldsymbol{M}(\boldsymbol{\theta }_{\boldsymbol{g}, \alpha })^{T}\boldsymbol{\Psi }_{n}\left( \boldsymbol{\theta }_{\boldsymbol{g}, \alpha }\right) ^{-1}\\
& \times \sqrt{n}\left( \frac{\partial H_{n, \alpha }(\boldsymbol{\theta })}{\partial \boldsymbol{\theta }}\right) _{\boldsymbol{\theta =\theta }_{\boldsymbol{g}, \alpha }}
+o_{p}(1).
	\end{aligned}
\end{equation*}

Finally we have,
\begin{equation*}
\sqrt{n}\left( \frac{\partial H_{n, \alpha }\left( \boldsymbol{\theta }\right) }{\partial \boldsymbol{\theta }^{T}}\right) _{\boldsymbol{\theta =\theta }_{\boldsymbol{g}, \alpha }}\underset{n\rightarrow \infty }{\overset{L}{\longrightarrow }}N(\boldsymbol{0}_{k},\boldsymbol{\Omega }_{n}\left( \boldsymbol{\theta }_{\boldsymbol{g}, \alpha }\right) ),
\end{equation*}
and thus the asymptotic distribution of $2n\left[ H_{n, \alpha }\left( \widehat{\boldsymbol{\theta }}_{\alpha }\right) - H_{n, \alpha }\left( \widetilde{\boldsymbol{\theta }}_{\alpha }\right) \right] $ coincides with the distribution of the random variable

\begin{equation*}
{\sum\limits_{i=1}^{r}}\lambda _{i}(\boldsymbol{\theta }_{\boldsymbol{g}, \alpha })Z_{i}^{2},
\end{equation*}
where $Z_{1},\ldots ,Z_{r}$ are independent standard normal variables, $\lambda _{1}(\boldsymbol{\theta }_{\boldsymbol{g}, \alpha }),\ldots ,\lambda _{r}(\boldsymbol{\theta }_{\boldsymbol{g}, \alpha })$ are the nonzero eigenvalues of $- \boldsymbol{Q}_{\alpha }(\boldsymbol{\theta }_{\boldsymbol{g}, \alpha })\boldsymbol{M}(\boldsymbol{\theta }_{\boldsymbol{g}, \alpha })^{T}\boldsymbol{\Psi }_{n}\left( \boldsymbol{\theta }_{\boldsymbol{g}, \alpha }\right) ^{-1}\boldsymbol{\Omega }_{n}\left( \boldsymbol{\theta }_{\boldsymbol{g}, \alpha }\right) $ and

\begin{equation*}
r = \mathrm{rank}\left( \boldsymbol{Q}_{\alpha }(\boldsymbol{\theta }_{\boldsymbol{g}, \alpha })\boldsymbol{M}(\boldsymbol{\theta }_{\boldsymbol{g}, \alpha })^{T}\boldsymbol{\Psi }_{n}\left( \boldsymbol{\theta }_{\boldsymbol{g}, \alpha }\right)^{-1}\boldsymbol{\Omega }_{n}\left( \boldsymbol{\theta }_{\boldsymbol{g}, \alpha }\right) \right) .
\end{equation*}

For more details see \ Corollary 2.1 in \cite{digu85}. This finishes the proof.
\end{proof}

The above result provides a way to asymptotically compute the probability of over-fitting, which is of great interest in model selection theory.


	\subsection{Example: The RP-based model selection under the multiple linear regression model and restricted parameter spaces.}
We shall consider the MLRM as defined in Section \ref{sec:RPMLRM} and we are interested in comparing a full model with a restricted model under the restrictions
\begin{equation*}
\beta _{p-r+1}=...=\beta _{p}=0.
\end{equation*}

In this case the model parameter is $\boldsymbol{\theta} = \left( \beta _{0},...,\beta _{p},\sigma \right) $ and the function $\boldsymbol{m}(\boldsymbol{\theta })$ defining the restrictions is
\begin{equation*}
\boldsymbol{m}(\boldsymbol{\theta })=\boldsymbol{m}\left( \beta_{0},...,\beta _{p},\sigma \right) =(\beta _{p-r+1},...,\beta _{p}).
\end{equation*}

Consequently, its derivative is given by
\begin{equation*}
\boldsymbol{M}(\boldsymbol{\theta })=\frac{\partial \boldsymbol{m}(\boldsymbol{\theta })}{\partial \boldsymbol{\theta }}=\left(
\begin{array}{c} \boldsymbol{0}_{(p-r+1)\times r} \\ \boldsymbol{I}_{r\times r} \\ \boldsymbol{0}_{1\times r} \end{array}
\right) .
\end{equation*}

Let us expressed the design matrix $\mathbb{X}$ as
\begin{equation*}
\mathbb{X=}\left( \mathbb{X}_{1}\text{,}\mathbb{X}_{2}\right) ,
\end{equation*}
with $\mathbb{X}_{1}$ a $n\times (p-r+1)$ matrix and $\mathbb{X}_{2}$ a $n\times r$ matrix. It is clear that $\mathbb{X}_{1}$ is the design matrix for the restricted model and $\mathbb{X}_{2}$ corresponds to the design matrix for the full model whose parameters are not in the small model. The matrices $\boldsymbol{\Psi }_{n}\left( \boldsymbol{\beta },\sigma \right) $ and $\boldsymbol{\Omega }_{n}\left( \boldsymbol{\beta },\sigma \right) $ given in Eq. (\ref{EqsSR}) can be rewritten, using the notation $\mathbb{X}_{1}$ and $\mathbb{X}_{2},$ as

\begin{equation*}
\boldsymbol{\Psi }_{n}\left( \boldsymbol{\beta },\sigma \right) = K_{1}\left( \alpha +1\right) ^{-\frac{3}{2}}\left[
\begin{array}{ccc}
\frac{1}{n}\mathbb{X}_{1}^{T}\mathbb{X}_{1} & \frac{1}{n}\mathbb{X}_{1}^{T}\mathbb{X}_{2} & 0 \\
\frac{1}{n}\mathbb{X}_{2}^{T}\mathbb{X}_{1} & \frac{1}{n}\mathbb{X}_{2}^{T}\mathbb{X}_{2} & 0 \\
0 & 0 & \frac{2}{\alpha +1}
\end{array}
\right] ,
\end{equation*}
being $K_{1}$ as defined in (\ref{K1}) and

\begin{equation*}
\boldsymbol{\Omega }_{n}\left( \boldsymbol{\beta },\sigma \right) =K_{1}^{2}\sigma ^{2}\frac{1}{\left( 2\alpha +1\right) ^{3/2}}\left[
\begin{array}{ccc}
\frac{1}{n}\mathbb{X}_{1}^{T}\mathbb{X}_{1} & \frac{1}{n}\mathbb{X}_{1}^{T}\mathbb{X}_{2} & 0 \\
\frac{1}{n}\mathbb{X}_{2}^{T}\mathbb{X}_{1} & \frac{1}{n}\mathbb{X}_{2}^{T}\mathbb{X}_{2} & 0 \\
0 & 0 & \frac{(3\alpha ^{2}+4\alpha +2)}{(\alpha +1)^{2}(2\alpha +1)}
\end{array}
\right] .
\end{equation*}

Now, the inverse of the matrix $\boldsymbol{\Psi }_{n}\left( \boldsymbol{\beta },\sigma \right) $ is given by
\begin{equation*}
\boldsymbol{\Psi }_{n}^{-1}\left( \boldsymbol{\beta },\sigma \right) =K_{1}\left( \alpha +1\right) ^{3/2}\left[
\begin{array}{ccc} n\boldsymbol{A}_{11} & n\boldsymbol{A}_{12} & 0 \\ n\boldsymbol{A}_{21} & n\boldsymbol{A}_{22} & 0 \\ 0 & 0 & \frac{\alpha +1}{2}\end{array}
\right] ,
\end{equation*}
with
\begin{eqnarray*}
\boldsymbol{A}_{11} & = & \left( \mathbb{X}_{1}^{T}\mathbb{X}_{1}\right)^{-1}+\left( \mathbb{X}_{1}^{T}\mathbb{X}_{1}\right) ^{-1}\mathbb{X}_{1}^{T}\mathbb{X}_{2}\boldsymbol{D}^{-1}\mathbb{X}_{2}^{T}\mathbb{X}_{1}\left(
\mathbb{X}_{1}^{T}\mathbb{X}_{1}\right) ^{-1}, \\
\boldsymbol{A}_{12} & = & -\left( \mathbb{X}_{1}^{T}\mathbb{X}_{1}\right) ^{-1}\mathbb{X}_{1}^{T}\mathbb{X}_{2}\boldsymbol{D}^{-1}, \\
\boldsymbol{A}_{21} & = & -\boldsymbol{D}^{-1}\mathbb{X}_{2}^{T}\mathbb{X}_{1}\left( \mathbb{X}_{1}^{T}\mathbb{X}_{1}\right) ^{-1}, \\
\boldsymbol{A}_{22} & = &\boldsymbol{D}^{-1},
\end{eqnarray*}
being
\begin{equation*}
\boldsymbol{D=}\mathbb{X}_{2}^{T}\mathbb{X}_{2}-\mathbb{X}_{2}^{T}\mathbb{X}_{1}\left( \mathbb{X}_{1}^{T}\mathbb{X}_{1}\right) ^{-1}\mathbb{X}_{1}^{T}\mathbb{X}_{2}.
\end{equation*}

Therefore, we have that the matrix $\boldsymbol{\Psi }_{n}^{-1}\left( \boldsymbol{\beta },\sigma \right)$ can be computed as
\begin{eqnarray*}
\boldsymbol{\Psi }_{n}^{-1}\left( \boldsymbol{\beta },\sigma \right) \boldsymbol{M}(\boldsymbol{\beta },\sigma ) &=&K_{1}^{-1}\left( \alpha +1\right)^{3/2}\left[
\begin{array}{ccc} n\boldsymbol{A}_{11} & n\boldsymbol{A}_{12} & 0 \\ n\boldsymbol{A}_{21} & n\boldsymbol{A}_{22} & 0 \\ 0 & 0 & \frac{\alpha +1}{2} \end{array}
\right] \left(
\begin{array}{c} \boldsymbol{0}_{(p-r)\times r} \\ \boldsymbol{I}_{r\times r} \\ \boldsymbol{0}_{r} \end{array}
\right)  \\
& = & K_{1}^{-1}\left( \alpha +1\right) ^{3/2}n\left[
\begin{array}{c} -\left( \mathbb{X}_{1}^{T}\mathbb{X}_{1}\right) ^{-1}\mathbb{X}_{1}^{T}\mathbb{X}_{2}\boldsymbol{D}^{-1} \\ \boldsymbol{D}^{-1} \\ 0 \end{array}
\right] .
\end{eqnarray*}

On the other hand,
\begin{equation*}
\left( \boldsymbol{M}(\boldsymbol{\beta },\sigma )^{T}\boldsymbol{\Psi }_{n}^{-1}\left( \boldsymbol{\beta },\sigma \right) \boldsymbol{M}(\boldsymbol{\beta },\sigma )\right) ^{-1}=\frac{K_{1}\left( \alpha +1\right) ^{-3/2}}{n}\boldsymbol{D},
\end{equation*}
and
\begin{equation*}
\boldsymbol{M}(\boldsymbol{\beta },\sigma )^{T}\boldsymbol{\Psi }_{n}^{-1}\left( \boldsymbol{\beta },\sigma \right) \boldsymbol{\Omega }_{n}\left( \boldsymbol{\beta },\sigma \right) =
\left( \alpha +1\right) ^{3/2}\frac{K_{1}\sigma ^{2}}{(2\alpha +1)^{3/2}}\left( \boldsymbol{0,\boldsymbol{I}_{r\times r},0}\right) ,
\end{equation*}
and so, multiplying the above expressions we obtain that
\begin{equation*}
\boldsymbol{Q}_{\alpha }(\boldsymbol{\beta },\sigma )=\boldsymbol{\Psi }_{n}^{-1}\left( \boldsymbol{\beta },\sigma \right) \boldsymbol{M}(\boldsymbol{\beta },\sigma )\left[ \boldsymbol{M}(\boldsymbol{\beta },\sigma )^{T}\boldsymbol{\Psi }_{n}^{-1}(\boldsymbol{\beta },\sigma )\boldsymbol{M}(\boldsymbol{\beta },\sigma )\right]^{-1} =\left[
\begin{array}{c} -\left( \mathbb{X}_{1}^{T}\mathbb{X}_{1}\right) ^{-1}\mathbb{X}_{1}^{T} \mathbb{X}_{2} \\ \boldsymbol{I}_{r\times r} \\ 0 \end{array}
\right] ,
\end{equation*}
and
\begin{equation*}
\boldsymbol{Q}_{\alpha }(\boldsymbol{\beta },\sigma )\boldsymbol{M}(\boldsymbol{\beta },\sigma )^{T}\boldsymbol{\Psi }_{n}\left( \boldsymbol{\beta },\sigma \right) ^{-1}\boldsymbol{\Omega }_{n}\left( \boldsymbol{\beta },\sigma \right) = \left( \alpha +1\right) ^{3/2}\frac{K_{1}\sigma_{\boldsymbol{g}, \alpha } ^{2}}{(2\alpha +1)^{3/2}}\left(
\begin{array}{ccc}
\mathbf{0} & \left( \mathbb{X}_{1}^{T}\mathbb{X}_{1}\right) ^{-1}\mathbb{X}_{1}^{T}\mathbb{X}_{2} & \mathbf{0} \\ \mathbf{0} & \boldsymbol{I}_{r\times r} & \mathbf{0} \\ \mathbf{0} & \mathbf{0} & \mathbf{0}
\end{array}
\right) .
\end{equation*}

Consequently, in this case we can compute the $r-$first eigenvalues as
\begin{equation*}
\lambda _{1}(\boldsymbol{\theta }_{\boldsymbol{g}, \alpha })=\ldots =\lambda _{r}(\boldsymbol{\theta }_{\boldsymbol{g}, \alpha })=\left( \alpha +1\right) ^{3/2}\frac{K_{1}\sigma _{\boldsymbol{g}, \alpha }^{2}}{(2\alpha +1)^{3/2}},
\end{equation*}
and hence,
$$ \sum_{i=1}^r \lambda_i(\boldsymbol{\theta }_{\boldsymbol{g}, \alpha }) Z_i^2 = - \left( \alpha +1\right) ^{3/2}\frac{K_{1}\sigma _{\boldsymbol{g}, \alpha }^{2}}{(2\alpha +1)^{3/2}} \chi^2_r.$$

On the other hand, we have
$$ \boldsymbol{\Omega }_{n}\left( \widehat{\boldsymbol{\beta }}_{\alpha }, \widehat{\sigma }_{\alpha }\right) \boldsymbol{\Psi }_{n}^{-1}\left( \widehat{\boldsymbol{\beta }}_{\alpha }, \widehat{\sigma }_{\alpha }\right) =  K_1 \sigma_{\boldsymbol{g}, \alpha }^2 {(\alpha +1)^{3/2}\over (2\alpha +1)^{3/2}}
\left[ \begin{array}{ccc}
\boldsymbol{I} & \boldsymbol{0} &  \boldsymbol{0} \\ \boldsymbol{0} & \boldsymbol{I} & \boldsymbol{0} \\  \boldsymbol{0} &  \boldsymbol{0} & {(3\alpha^2 + 4\alpha +2) \over 2(2\alpha +1)(\alpha +1)}
\end{array} \right] ,$$
and hence the trace of the above matrix is given by
\begin{equation*}
trace\left( \boldsymbol{\Omega }_{n}\left( \widehat{\boldsymbol{\beta }}_{\alpha }, \widehat{\sigma }_{\alpha }\right) \boldsymbol{\Psi }_{n}^{-1}\left( \widehat{\boldsymbol{\beta }}_{\alpha }, \widehat{\sigma }_{\alpha }\right) \right) \rightarrow
\sigma _{\boldsymbol{g}, \alpha }^{2}K_{1} {(\alpha +1)^{3/2}\over (2\alpha +1)^{3/2}} \left( (p+1) +\frac{\left( 3\alpha ^{2}+4\alpha +2\right) }{2\left( 2\alpha +1\right) \left( \alpha +1\right) }\right) ,
\end{equation*}
and

\begin{equation*}
trace\left( \boldsymbol{\Omega }_{n}^R\left( \widetilde{\boldsymbol{\beta }}_{\alpha },\widetilde{\sigma }_{\alpha }\right) \right) (\boldsymbol{\Psi }_{n}^{R})^{-1}\left( \widetilde{\boldsymbol{\beta }}_{\alpha },\widetilde{\sigma }_{\alpha }\right) \rightarrow
\sigma _{\boldsymbol{g}, \alpha }^{2}K_{1}{(\alpha +1)^{3/2}\over (2\alpha +1)^{3/2}} \left( \left( p-r+1\right) +\frac{\left( 3\alpha ^{2}+4\alpha +2\right) }{2\left( 2\alpha +1\right) \left( \alpha +1\right) }\right) .
\end{equation*}

Therefore,
\begin{equation*}
trace\left( \boldsymbol{\Omega }_{n}\left( \widehat{\boldsymbol{\beta }}_{\alpha },\widehat{\sigma }_{\alpha }\right) \boldsymbol{\Psi }_{n}^{-1}\left( \widehat{\boldsymbol{\beta }}_{\alpha },\widehat{\sigma }_{\alpha }\right) \right) - trace\left( \boldsymbol{\Omega }_{n}^R\left( \widetilde{\boldsymbol{\beta }}_{\alpha }, \widetilde{\sigma }_{\alpha }\right) (\boldsymbol{\Psi }_{n}^{R})^{-1}\left( \widetilde{\boldsymbol{\beta }}_{\alpha },\widetilde{\sigma }_{\alpha }\right) \right) \rightarrow
 \sigma_{\boldsymbol{g}, \alpha }^{2}K_{1}{(\alpha +1)^{3/2}\over (2\alpha +1)^{3/2}} r.
\end{equation*}

Finally, the asymptotic probability of selecting the restricted model when this model is correct is

$$ \Pr \left( 2n\left( RP_{NH} (M^{(s)}_{1},..., M^{(s)}_{n}, \widehat{\boldsymbol{\theta }}_{\alpha }) -RP_{NH}(M^{(s)}_{1},..., M^{(s)}_{n}, \widetilde{\boldsymbol{\theta }}_{\alpha })\right) >0\right) \rightarrow $$
$$ \Pr \left( \left( -\alpha +1\right) ^{3/2}\frac{K_{1}\sigma _{\boldsymbol{g}, \alpha }^{2}}{(2\alpha +1)^{3/2}}\chi _{r}^{2}+ 2 \left( \alpha +1\right) ^{3/2}\frac{K_{1}\sigma _{\boldsymbol{g}, \alpha }^{2}}{(2\alpha +1)^{3/2}}r>0\right) = $$ $$=\Pr \left( \left( \alpha +1\right) ^{3/2}\frac{K_{1}\sigma _{\boldsymbol{g}, \alpha }^{2}}{(2\alpha +1)^{3/2}} \left( 2r- \chi _{r}^{2} \right) >0 \right) = \Pr \left( \chi _{r}^{2}<2 r \right) .$$


\section{Simulation Study}

To evaluate the performance of the $RP_{NH}$-criterion introduced in this paper, we consider the situation of a polynomial regression model. We take the model

$$ Y_i= X_i+2X_i^2  -X_i^3 + X_i^4 + \epsilon_i, i=1, ..., n,$$
where $\epsilon_i \sim {\cal  N}(0,1)$ and the variables $X_i$ are fixed and chosen uniformly in the interval [-2, 2]. Next, we take $n=100$, so that

$$X_i = -2 + {4\over 102}(i+1), i=1, ..., 100.$$

We consider several theoretical models aiming to fit this data. These models are given by the degree of the polynomial defining the model. Note that the regression coefficients adopt the same expression as in MLRM, just taking $X^i$ as $X_i,$ and thus we can use the formulas developed in the previous sections. In our case, we have considered six different models, varying from constants (degree 0) to polynomials of degree 5. Thus defined, each model is characterized by the degree, denoted by $p$.

We take 1000 different sample data $(Y^s, X^s), s=1, ..., 1000$ and for each sample, we select the best fitting model according to several criteria. We have considered $AIC, BIC, AIC_c$ and the $RP_{NH}$-criterion for different values of the tuning parameter, namely $\alpha =0.01, 0.02, 0.04, 0.07, 0.1, 0.2, 0.4, 0.5, 0.7$ and 1.

In Table \ref{tabsim1} we have written the number of times that each model is selected for each model selection criterion. From these results, it can be seen that $BIC$ seems to be the best fitting selection criterion, the other model selection criteria having a similar performance.

\begin{table}[h]
\begin{center}
\begin{tabular}{|c|cccccc|}
\hline $p$   &  0 &   1 &   2 &   3 &    4 &   5 \\
\hline $AIC$ & 0 &   0 &  0 &   0 & 836  & 164  \\
$BIC$        & 0   & 0   & 0  &  0  & 967  & 33 \\
$AIC_c$      & 0   & 0   & 0  &  0  & 864  & 136 \\
$RPNH_{0.01}$& 0   & 0   & 0  &  0  & 822  & 178 \\
$RPNH_{0.02}$& 0   & 0   & 0  &  0  & 822  & 178 \\
$RPNH_{0.04}$& 0   & 0   & 0  &  0  & 826  & 174 \\
$RPNH_{0.1}$& 0   & 0   & 0  &  0  & 822  & 178 \\
$RPNH_{0.2}$& 0   & 0   & 0  &  0  & 834  & 166 \\
$RPNH_{0.4}$& 0   & 0   & 0  &  0  & 842  & 158 \\
$RPNH_{0.5}$& 0   & 0   & 0  &  0  & 841  & 159 \\
$RPNH_{0.7}$& 0   & 0   & 0  &  0  & 838  & 162 \\
$RPNH_{1.0}$& 0   & 0   & 0  &  0  & 837  & 163 \\
\hline
\end{tabular}
\caption{Results for uncontaminated data.}
\label{tabsim1}
\end{center}
\end{table}

As it was explained throughout the paper, we expect $RP_{NH}$ to be a robust selection criterion. To check this hypothesis, we have considered a situation of contamination. Thus, we consider the previous model but we introduce contamination in some of the data. More concretely, we define

$$ \epsilon_i \sim {\cal U}(\min_i (X_i+2X_i^2  -X_i^3 + X_i^4) -r, \max_i (X_i+2X_i^2  -X_i^3 + X_i^4) + r),$$
for some of the data chosen at random. Here, $r$ is a constant measuring the strength of contamination, in the sense that the bigger $r,$ the strongest the contamination. We have considered three valus $r=1, 5, 10.$ Moreover, we have varied the proportion of data affected by contamination. In this study, we have chosen the proportion of contamination as $0.05, 0.10, 0.20, 0.30.$

Again, we have obtained the best fitting model according different model selection criteria, and we have conducted this experiment 1000 times. The number of times that each model is selected for each combination of contamination and strength of contamination $r$ is given in Tables \ref{tabsim2}, \ref{tabsim3}, \ref{tabsim4} and \ref{tabsim5}. The left part of each table corresponds to $r=1,$ the center part for $r=5$ and the right part for $r=10.$

\begin{table}[h]
\begin{center}
	
\resizebox{\textwidth}{!}{
\begin{tabular}{|c|cccccc|cccccc|cccccc|}
	\hline
		& \multicolumn{6}{c}{$r=1$} & \multicolumn{6}{c}{$r=5$} & \multicolumn{6}{c|}{$r=10$} \\
\hline $p$  &  0 &   1 &   2 &   3 &    4 &   5 & 0 &   1 &   2 &   3 &    4 &   5 & 0 &   1 &   2 &   3 &    4 &   5\\
\hline $AIC$ & 0  & 0   & 1  & 19  & 659  & 321 & 0 &   0 &  10 &  27 & 622  &341  & 0 &   0 &   9 &  38 & 616  & 337 \\
$BIC$        & 0  & 0   & 16 & 64  & 802  & 118 & 0 &   0 &  39 &  84 & 751  & 126 & 0 &   0 &  57 &  96 & 715  & 132 \\
$AIC_c$      & 0  & 0   & 1  & 22  & 694  & 283 & 0 &   0 &  12 &  33 & 651  & 304 & 0 &   0 &  12 &  43 & 634  & 311 \\
$RPNH_{0.01}$ & 0  &  0  &  3 &  14 & 844  &139  & 0  &  0  &  5 &  18 & 830  & 147 & 0  &  0  &  9 &  22 & 812  &157 \\
$RPNH_{0.02}$ &  0 &  0  &  0 &  13 & 866  & 121 & 0  &  0  &  2 &  47 & 833  & 118 & 0  &  0  &  6 &  62 & 810  &122 \\
$RPNH_{0.04}$ &  0 &   0 &   2&   20&  844 & 134 & 0  &  0  &  1 &  18 & 833  & 148 & 0  &  0  &  1 &  18 & 833  &148 \\
$RPNH_{0.1}$ & 0  & 0   & 0  &  0  & 835  & 165 & 0 &   0 &   0 &   0 & 836  &164  & 0 &   0 &   0 &   0 & 830  & 170\\
$RPNH_{0.2}$ & 0  & 0   & 0  &  0  & 829  & 171 & 0 &   0 &   0 &   0 & 833  & 167 & 0 &   0 &   0 &   0 & 833  & 167\\
$RPNH_{0.4}$ & 0  & 0   & 0  &  0  & 837  & 163 & 0 &   0 &   0 &   0 & 835  & 165 & 0 &   0 &   0 &   0 & 839  & 161 \\
$RPNH_{0.5}$ & 0  & 0   & 0  &  0  & 837  & 163 & 0 &   0 &   0 &   0 & 848  & 152 & 0 &   0 &   0 &   0 & 834  & 166\\
$RPNH_{0.7}$ & 0  & 0   & 0  &  0  & 842  & 158 & 0 &   0 &   0 &   0 & 836  & 164 & 0 &   0 &   0 &   0 & 831  & 169\\
$RPNH_{1.0}$ & 0  & 0   & 0  &  0  & 838  & 162 & 0 &   0 &   0 &   0 & 836  & 164 & 0 &   0 &   0 &   0 & 830  & 170\\
\hline
\end{tabular}
}
\caption{Results for contamination degree of $5\%$}
\label{tabsim2}
\end{center}
\end{table}

\begin{table}[h]
\begin{center}
\resizebox{\textwidth}{!}{
\begin{tabular}{|c|cccccc|cccccc|cccccc|}
	\hline
		& \multicolumn{6}{c}{$r=1$} & \multicolumn{6}{c}{$r=5$} & \multicolumn{6}{c|}{$r=10$} \\
\hline $p$  &  0 &   1 &   2 &   3 &    4 &   5 &  0 &   1 &   2 &   3 &    4 &   5 &  0 &   1 &   2 &   3 &    4 &   5\\
\hline $AIC$ & 0  & 0   & 17 & 64  & 591  & 328 & 0 &   0 &  18 &  82 & 575 & 325   & 0  &  0  & 47  & 106 & 538  & 309\\
$BIC$        & 0  & 0   & 94 & 155 & 629  & 122 & 0 &   0 &  95 & 180 & 611 & 114   & 0  &  0  &153  & 178 & 558  & 111\\
$AIC_c$      & 0  & 0   & 21 &  75 & 621  & 283 & 0 &   0 &  24 &  93 & 601 & 282   & 0  &  0  & 55  & 123 & 556  & 266\\
$RPNH_{0.01}$ &  0 &   0 &  24&   37&  770 & 169 & 0  &  0  & 26 &  48 & 750  & 176 & 0  &  0  & 37 &  61 & 705  &197 \\
$RPNH_{0.02}$ &  0 &   0 &  19&   30&  800 & 151 & 0  &  0  & 24 &  40 & 780  & 156 & 0  &  0  & 30 &  60 & 747  &163 \\
$RPNH_{0.04}$ &  0 &   0 &  16&   60&  809 & 115 & 0  &  0  & 19 & 100 & 764  & 117 & 0  &  0  & 23 &  94 & 770  &113 \\
$RPNH_{0.1}$ & 0  & 0   & 0  &  5  & 845  & 150 & 0 &   0 &   0 &   1 & 839 & 160   & 0  &  0  &  0  &  1  & 851  & 148\\
$RPNH_{0.2}$ & 0  & 0   & 0  &  0  & 829  & 171 & 0 &   0 &   0 &   0 & 835 & 165   & 0  &  0  &  0  &  0  & 844  & 156\\
$RPNH_{0.4}$ & 0  & 0   & 0  &  0  & 829  & 171 & 0 &   0 &   0 &   0 & 840 & 160   & 0  &  0  &  0  &  0  & 850  & 150 \\
$RPNH_{0.5}$ & 0  & 0   & 0  &  0  & 824  & 176 & 0 &   0 &   0 &   0 & 845 & 155   & 0  &  0  &  0  &  0  & 841  & 159 \\
$RPNH_{0.7}$ & 0  & 0   & 0  &  0  & 831  & 169 & 0 &   0 &   0 &   0 & 833 & 167   & 0  &  0  &  0  &  0  & 834  & 166 \\
$RPNH_{1.0}$ & 0  & 0   & 0  &  0  & 841  & 159 & 0 &   0 &   0 &   0 & 835 & 165   & 0  &  0  &  0  &  0  & 833  & 167\\
\hline
\end{tabular}
}
\caption{Results for a contamination degree of $10\%$.}
\label{tabsim3}
\end{center}
\end{table}

\begin{table}[h]
\begin{center}
	\resizebox{\textwidth}{!}{
\begin{tabular}{|c|cccccc|cccccc|cccccc|}
	\hline
		& \multicolumn{6}{c}{$r=1$} & \multicolumn{6}{c}{$r=5$} & \multicolumn{6}{c|}{$r=10$} \\
\hline $p$  &  0 &   1 &   2 &   3 &    4 &   5 &  0 &   1 &   2 &   3 &    4 &   5 &  0 &   1 &   2 &   3 &    4 &   5 \\
\hline $AIC$ & 0  & 0   & 52 & 134 & 509  & 305 & 0  &  0  & 76  & 176 & 464  & 284 &  0 &   0 & 148 & 169 & 397  & 286 \\
$BIC$        & 0  & 0   &238 & 210 & 440  & 112 & 0  &  0  &278  & 234 & 398  & 90  & 0  &  0  & 367 & 223 & 325  & 85\\
$AIC_c$      & 0  & 0   & 64 & 154 & 512  & 270 & 0  &  0  & 91  & 191 & 476  & 242 & 0  &  0  & 179 & 180 & 400  & 241\\
$RPNH_{0.01}$ & 0  &  0  & 41 &  92 & 596  & 271 & 0  &  0  & 43 &  93 & 561  & 303 & 0  &  0  & 52 &  95 & 514  & 339 \\
$RPNH_{0.02}$ & 0  &  0  & 37 &  85 & 625  & 253 & 0  &  0  & 39 &  92 & 589  & 280 & 0  &  0  & 47 &  90 & 561  & 302 \\
$RPNH_{0.04}$ & 0  &  0  & 29 &  75 & 676  & 220 & 0  &  0  & 32 &  88 & 661  & 219 & 0  &  0  & 43 &  93 & 648  & 216 \\
$RPNH_{0.1}$ & 0  & 0   & 42 & 214 & 631  & 113 & 0  &  0  & 41  & 138 & 693  & 128 & 0  &  0  & 20  & 64  & 810  & 106\\
$RPNH_{0.2}$ & 0  & 0   & 0  &  0  & 836  & 164 & 0  &  0  &  0  &  0  & 831  & 169 & 0  &  0  &  0  &  0  & 849  & 151\\
$RPNH_{0.4}$ & 0  & 0   & 0  &  0  & 837  & 163 & 0  &  0  &  0  &  0  & 833  & 167 & 0  &  0  &  0  &  0  & 840  & 160\\
$RPNH_{0.5}$ & 0  & 0   & 0  &  0  & 836  & 164 & 0  &  0  &  0  &  0  & 829  & 171 & 0  &  0  &  0  &  0  & 840  & 160\\
$RPNH_{0.7}$ & 0  & 0   & 0  &  0  & 845  & 155 & 0  &  0  &  0  &  0  & 827  & 173 & 0  &  0  &  0  &  0  & 849  & 151\\
$RPNH_{1.0}$ & 0  & 0   & 0  &  0  & 834  & 166 & 0  &  0  &  0  &  0  & 823  & 177 & 0  &  0  &  0  &  0  & 836  & 164\\
\hline
\end{tabular}
}
\caption{Results for a contamination degree of $20\% $.}
\label{tabsim4}
\end{center}
\end{table}

\begin{table}[h]
\begin{center}
	\resizebox{\textwidth}{!}{
\begin{tabular}{|c|cccccc|cccccc|cccccc|}
	\hline
		& \multicolumn{6}{c}{$r=1$} & \multicolumn{6}{c}{$r=5$} & \multicolumn{6}{c|}{$r=10$} \\
\hline $p$  &  0 &   1 &   2 &   3 &    4 &   5 &  0 &   1 &   2 &   3 &    4 &   5 &  0 &   1 &   2 &   3 &    4 &   5\\
\hline $AIC$ & 0  & 0   &112 & 178 & 433  &277  & 0  &  0  & 114 & 209 & 373  & 304 & 0  &  0  & 192 & 183 & 374  & 251\\
$BIC$        & 0  & 0   &327 & 256 & 327  & 90  & 0  &  0  & 385 & 240 & 276  & 99  & 0  &  0  & 457 & 212 & 253  & 78\\
$AIC_c$      & 0  & 0   &136 & 191 & 436  & 237 & 0  &  0  & 137 & 233 & 368  & 262 & 0  &  0  & 219 & 189 & 371  & 221\\
$RPNH_{0.01}$ & 0  &  0  & 51 &  79 & 519  &351  & 0  &  0  & 55 &  90 & 488  & 367 & 0  &  0  & 58 &  90 & 428  & 424 \\
$RPNH_{0.02}$ & 0  &  0  & 46 &  77 & 540  &337  & 0  &  0  & 48 &  84 & 520  &348  & 0  &  0  & 52 &  87 & 472  & 389 \\
$RPNH_{0.04}$ & 0  &  0  & 44 &  81 & 573  &302  & 0  &  0  & 46 &  80 & 555  & 319 & 0  &  0  & 53 &  78 & 533  & 336 \\
$RPNH_{0.1}$ & 0  & 0   & 70 & 187 & 550  & 193 & 0  &  0  & 63  & 221 & 537  & 179 & 0  &  0  & 55  & 139 & 628  & 178\\
$RPNH_{0.2}$ & 0  & 0   & 17 &  68 & 774  & 141 & 0  &  0  & 11  & 13  & 817  & 159 & 0  &  0  &  2  &  8  & 854  & 136\\
$RPNH_{0.4}$ & 0  & 0   & 0  &  0  & 856  & 144 & 0  &  0  &  0  &  0  & 833  & 167 & 0  &  0  &  0  &  0  & 841  & 159\\
$RPNH_{0.5}$ & 0  & 0   & 0  &  0  & 845  & 155 & 0  &  0  &  0  &  0  & 830  & 170 & 0  &  0  &  0  &  0  & 832  & 168\\
$RPNH_{0.7}$ & 0  & 0   & 0  &  0  & 834  & 166 & 0  &  0  &  0  &  0  & 815  & 185 & 0  &  0  &  0  &  0  & 826  & 174\\
$RPNH_{1.0}$ & 0  & 0   & 0  &  0  & 828  & 172 & 0  &  0  &  1  &  1  & 813  & 185 & 0  &  0  &  0  &  0  & 841  & 159\\
\hline
\end{tabular}
}
\caption{Results for a contamination degree of $30\%$.}
\label{tabsim5}
\end{center}
\end{table}

From the results in these tables, it can be seen that the performance of $AIC, BIC$ and $AIC_c$ dramatically decrease, in the sense that the proportion of times obtaining the true degree $p=4$ decreases if contamination is present. As expected, the bigger the rate of contaminated data, the poorer the performance. Note however that they are not very affected for different values of $r$.

On the other hand, the results are quite similar to the uncontaminated case for $RP_{NH}$ and big values of the tuning parameter. This was the expected result and it follows the same behavior as other situations where RP has been considered. The best behavior appears for $\alpha =0.4$ and $\alpha =0.5,$ where the efficiency is good and the performance in terms of robustness is very good.

Finally, in order to test if this is the usual behavior of these methods, we have repeated this study for different values of the polynomial regression, each coefficient varying in $\{ -2, -1, 0, 1, 2\} .$ This leads to 3125 different models for each value of $r= 1, 5, 10,$ so that we have 9 375 different situations. And for all of them we can extract the same conclusions.

\section{Real data example}

In this section we analyze a set of real data at the light of this new model selection tool based on RP. We consider the problem proposed in \cite{drsm81} and later studied in \cite{tokatr20}. The dependent variable $Y$ measures the heat evolved in calories per gram as a function of four ingredients: tricalcium aluminate ($X_1$), tricalcium silicate ($X_2$), tetracalcium alumino-ferrite ($X_3$) and dicalcium silicate ($X_4$). The data are given in Table \ref{tableEx1}. It is assumed that $Y$ can be written in terms of $X_1, X_2, X_3, X_4$ as a MLRM. We have considered the $RP_{NH}$ procedure to select the best model for different values of the tuning parameter.

\begin{table}[h]
\begin{center}
\begin{tabular}{|cccc|c|}
\hline $X_1$ & $X_2$ & $X_3$ & $X_4$ & $Y$ \\
\hline 7     & 26    & 6     & 60    & 78.5 \\
       1     & 29    & 15    & 52    & 74.3 \\
      11     & 56    & 8     & 20    & 104.3 \\
      11     & 31    & 8     & 47    & 87.6 \\
       7     & 52    & 6     & 33    & 95.9 \\
      11     & 55    & 9     & 22    & 109.2 \\
       3     & 71    & 17    & 6     & 102.7 \\
       1     & 31    & 22    & 44    & 72.5 \\
       2     & 54    & 18    & 22    & 93.1 \\
      21     & 47    & 4     & 26    & 115.9 \\
       1     & 40    & 23    & 34    & 83.8 \\
      11     & 66    & 9     & 12    & 113.3 \\
      10     & 68    & 8     & 12    & 109.4 \\
\hline
\end{tabular}
\caption{The Hald cement data.}
\label{tableEx1}
\end{center}
\end{table}

Considering different subsets of independent variables, we obtain 15 different multiple linear models and  the goal is to select the best one. However, it is known that at least two independent variables are needed because cement needs a combination of at least two reactants. Hence, we can remove the four simple linear regression models and we finally consider 11 possible models.

\begin{table}[h]
\begin{center}
\resizebox{\textwidth}{!}{
\begin{tabular}{|c|ccccc|}
\hline               & $RPNH_{0.01}$ & $RPNH_{0.02}$ & $RPNH_{0.04}$ & $RPNH_{0.05}$ & $RPNH_{0.07}$ \\
\hline    $X_1, X_2$ &  2.4179    &  2.3655        & 2.2642      & 2.2170      & 2.1280 \\
          $X_1, X_3$ &  3.8824    &  4.3871        & 4.1321      & 4.0138      & 3.7933 \\
          $X_1, X_4$ &  2.5408    &  2.4827        & 2.3738      & 2.3226      & 2.2261 \\
          $X_2, X_3$ &  3.3638    &  3.2836        & 3.1140      & 3.0349      & 2.8868 \\
          $X_2, X_4$ &  3.7164    &  3.8865        & 3.6579      & 3.5523      & 3.3565 \\
          $X_3, X_4$ &  2.9519    &  2.8785        & 2.7413      & 2.6771      & 2.5564 \\
$X_1, X_2, X_3$      & 2.4024     &  2.3493        & 2.2495      & 2.2026      & 2.1141 \\
$X_1, X_2, X_4$      & 2.4013     &  2.3484        & 2.2490      & 2.2023      & 2.1142 \\
$X_1, X_3, X_4$      & 2.4295     &  2.3759        & 2.2752      & 2.2279      & 2.1386 \\
$X_2, X_3, X_4$     &  2.6091     &  2.5490        & 2.4363      & 2.3834      & 2.2837 \\
$X_1, X_2, X_3, X_4$ & 2.4747     &  2.4197        & 2.3164      & 2.2679      & 2.1765 \\
\hline Best model    & $(X_1, X_2, X_4)$ & $(X_1, X_2, X_4)$ & $(X_1, X_2, X_4)$ & $(X_1, X_2, X_4)$ &$(X_1, X_2, X_3)$ \\
\hline               &                   &                   &                   &                   &                  \\
\hline               & $RPNH_{0.1}$ & $RPNH_{0.2}$ & $RPNH_{0.4}$ & $RPNH_{0.5}$ & $RPNH_{0.7}$ \\
\hline    $X_1, X_2$ &  2.0064    &  1.6822        & 1.2656      & 1.1249      & 0.9197 \\
          $X_1, X_3$ &  3.4984    &  2.7476        & 1.8439      & 1.5662      & 1.2019 \\
          $X_1, X_4$ &  2.0946    &  1.7454        & 1.3004      & 1.1517      & 0.9411 \\
          $X_2, X_3$ &  2.6873    &  2.1710        & 1.5474      & 1.3482      & 1.0702 \\
          $X_2, X_4$ &  3.0967    &  2.4472        & 1.7055      & 1.4771      & 1.1616 \\
          $X_3, X_4$ &  2.3930    &  1.9647        & 1.4322      & 1.2638      & 1.0347 \\
$X_1, X_2, X_3$      & 1.9933     &  1.6716        & 1.2598      & 1.1264      & 0.9173 \\
$X_1, X_2, X_4$      & 1.9939     &  1.6735        & 1.2623      & 1.1240      & 0.9228 \\
$X_1, X_3, X_4$      & 2.0167     &  1.6921        & 1.2756      & 1.1352      & 0.9985 \\
$X_2, X_3, X_4$     &  2.1482     &  1.7891        & 1.3340      & 1.1824      & 1.0089 \\
$X_1, X_2, X_3, X_4$ & 2.0521     &  1.7221        & 1.3014      & 1.1601      & 0.9548 \\
\hline Best model    & $(X_1, X_2, X_3)$ & $(X_1, X_2, X_3)$ & $(X_1, X_2, X_3)$ & $(X_1, X_2, X_4)$ &$(X_1, X_2, X_3)$ \\
\hline
\end{tabular}
}
\caption{Results for the Hald cement data.}
\label{tableEx1sol}
\end{center}
\end{table}

We have applied the $RP_{NH}$-criterion defined in (\ref{ali}) for different values of the tuning parameter to select the most appropriate model. The solution is given in Table \ref{tableEx1sol}. As it can be seen in this table, the combinations of $X_1, X_2, X_3$ and $X_1, X_2, X_4$ seem to be the best candidates, with tiny differences between them. These results are similar to the conclusions obtained in \cite{tokatr20}. Remark also the good performance of model $X_1, X_2.$

%
%
%
%

\section{Conclusions}

In this paper we have developed a new procedure for model selection for independent but not identically distributed observations aiming to compete with other methods based on maximum likelihood in terms of efficiency but being more robust agaisnt outlying data. For this purpose, we have considered RP, a tool that has proved itself to provide robust estimations in many statistical problems. We have developed a model selection criterion, the $RP_{NH}$-criterion, extending the well-known AIC. Besides, we have shown that the sample estimator is an unbiased estimator. Next, we have considered the case of having a restricted model and we have developed a procedure to decide whether the large model is more appropriate for modeling the available data. As an example of application, we have developed the MLRM when we aim to find the best model fitting a set of data. We have conducted a simulation study that shows that this new procedure works very well under contamination, i.e. simulations suggest that the procedure is robust. Besides, it seems that the cost in terms of efficiency is reduced. Finally, we have applied this new procedure in a situation with real data.

\section*{Acknowledgements}

This work was supported by the Spanish Grant PID2021-124933NB-I00.



\end{document}